\documentclass{Dissertate}
\usepackage{setspace}  
\usepackage{mytitlepage}
\usepackage{enumitem}
\usepackage[utf8]{inputenc}
\usepackage{graphicx, tikz, caption, subcaption, float, amsfonts, amsmath, amsthm, amssymb, subcaption, wrapfig, multicol}
\graphicspath{ {images/} }
\usepackage{pgfplots}
\pgfplotsset{compat=1.15}
\usepackage[numbers]{natbib}
\usepackage{algorithm}
\usepackage{algpseudocode}
\tikzset{>=stealth, node distance=1.5cm, main node/.style={circle,fill,inner sep=2pt}}

\usepackage{pgfplots}
\usepgfplotslibrary{patchplots}
\usetikzlibrary{patterns, positioning, arrows}
\pgfplotsset{compat=1.15}

\usepackage{fancyhdr} 
\usepackage{setspace, xcolor, ifmtarg, xifthen, environ, multido}
\usetikzlibrary{positioning, arrows}

\theoremstyle{definition}
\newtheorem{theorem}{Theorem}[section]
\newtheorem{proposition}[theorem]{Proposition}

\theoremstyle{definition} 
\newtheorem{lemma}[theorem]{Lemma}
\newtheorem{claim}[theorem]{Claim}

\theoremstyle{definition}
\newtheorem{definition}[theorem]{Definition}

\newtheorem{example}[theorem]{Example}
\newtheorem{remark}[theorem]{Remark}

\makeatletter
\providecommand{\thetitle}{\@title}
\providecommand{\theauthor}{\@author}
\makeatother
\title{Ollivier-Ricci Curvature of Riemannian Manifolds and Directed Graphs with Applications to Graph Neural Networks}

\providecommand{\TitleForTitlePage}{}
\renewcommand{\TitleForTitlePage}{%
  \begin{spacing}{1.0}
    \centering
    \Large
    Ollivier-Ricci Curvature of\par\vspace{0.25em}
    Riemannian Manifolds and Directed Graphs\par\vspace{0.25em}
    with Applications to Graph Neural Networks
  \end{spacing}%
}

\author{Eleanor Pyle Wiesler}

\providecommand{\advisors}{}
\renewcommand{\advisors}{Professor Melanie Weber\\Professor Laura DeMarco}

\degree{Bachelor of Arts with Honors}
\field{Mathematics}
\degreemonth{March}
\degreeyear{2025}
\department{Mathematics}

\begin{document}



\leavevmode\frontmatter
\setstretch{\dnormalspacing}


\setcounter{chapter}{-1}  
\begin{savequote}[90mm]
There is no branch of mathematics, however abstract, which may not someday be applied to the phenomena of the real world.
\qauthor{Nikolai Lobachevsky}
\end{savequote}

\chapter{Introduction}

Curvature is not limited to surfaces or manifolds and graphs are not limited to being undirected. This thesis is an exploration of a curvature notion defined for general metric spaces such as Riemannian manifolds and graphs called Ollivier-Ricci curvature. Introduced by Yann Ollivier and later extended by Shing-Tung Yau with colleagues Yong Lin and Linyuan Lu, this form of curvature builds upon the curvature tensor introduced by Gregorio Ricci-Curbastro in the 17th century through the contraction of the Riemann curvature tensor. Unlike classical Ricci curvature which depends on tensor calculus and is defined for Riemannian manifolds, the Ollivier-Ricci curvature between two points depends solely on the ratio between an optimal transport cost metric and the distance between the two points, allowing application to new metric spaces and non-derivative computation on Riemannian manifolds. The goal of this thesis is to present a rigorous exposition on how Ollivier-Ricci curvature both extends and modifies classical Ricci curvature for application on general metric spaces such as graphs.

Ollivier-Ricci curvature is defined simply as follows:

$$
\operatorname{Ric}(x,y) := 1 - \frac{W_1(m_x,m_y)}{d(x,y)}
$$

\noindent 
where $W_1$ is the 1-Wasserstein optimal transport distance between probability measures $m_x$ and $m_y$ and $d(x,y)$ is the distance between two points $x$ and $y$ in some metric space. We will discuss in detail these terms and this novel form of Ricci curvature in beginning of thesis in Chapters 1 and 2. By any reader familiar with differential or Riemannian geometry, this definition of curvature bears no resemblance to  classical Ricci curvature. This is because Ollivier did not derive this notion from the Ricci tensor but instead synthetically defined it such that it preserves various properties when negative, positive, and zero that are characteristic of negative, positive, and zero Ricci curvature.

\begin{figure}
    \centering
    \includegraphics[width=0.9\linewidth]{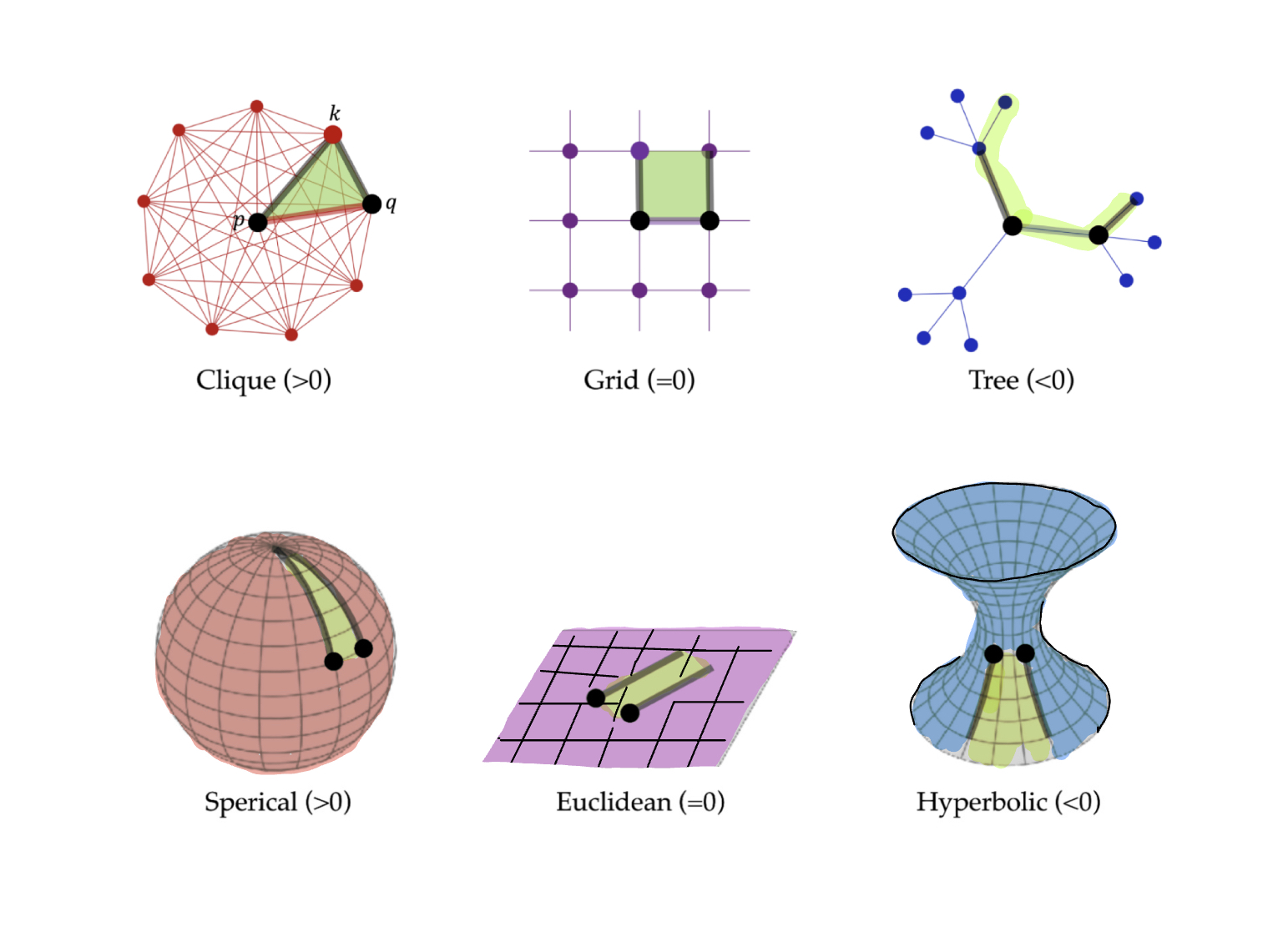}
    \caption{Figure adapted from Michael Bronstein (Over-squashing, bottlenecks, and graph Ricci curvature) representing negative, zero, and positive Ricci curvature analogies between manifolds and graphs. Ricci curvature captures geodesic dispersion, which can be modeled by assessing the behavior of parallel geodesics.}
    \label{fig:bronstein-ricci-analogy}
\end{figure}

Various questions tend to arise by those who are newly acquainted with Ollivier-Ricci curvature. In particular, regarding Ollivier-Ricci curvature's highly studies extension to graphs, one such question is often this: why is Ricci curvature used and not some other form of curvature? Indeed there are other forms of curvature which we will discuss: from Riemannian to sectional to scalar, or Gaussian curvature, there are other forms of both highly related and particularly distinct forms of curvature that are well studied aside from Ricci. The answer to this question is that Ricci curvature is of particular interest on manifolds, metric spaces, and and graphs for the study of localized curvature from a single point or on a single edge, for example. In particular, though, the Ricci tensor captures the extent to which the geodesic ball in a curved Riemannian manifold deviates from that of a ball in Euclidean space. Ricci curvature controls geodesic dispersion and volume growth, and this control or bounding of such growth that Ricci curvature provides allows for natural extension to graphs. A demonstrated in Figure 1, positive Ricci curvature is characterized by convergence of the parallel geodesics as they move along a manifold. Analogously on graphs of positive Ricci curvature this is captured by a highly connected graph cluster. Locally flat or Euclidean surface will have zero curvature by nature of the geodesics remaining parallel, which is also captured by a grid. Negative Ricci curvature of hyperbolic surfaces in which parallel geodesics diverge is captured analogously by the geodesic dispersion in trees. 

It is important to note that there are other forms of Ricci curvature that have been developed in recent years with some notable example being Bakry-Emery as an extension of Ricci curvature on smooth manifolds, and Forman's Ricci curvature defined for cell complexes. \cite{najman_modern_2017} This thesis focuses only on Ollivier-Ricci curvature and it's relation to classical Ricci curvature on a Riemannian manifold. The outline of this thesis chapter-by-chapter is now presented.

\section*{Thesis Outline}

\textit{Chapter 1: The 1-Wasserstein Distance}\\

This chapter functions to briefly introduce the major terminology and background required to understand the 1-Wasserstein distance metric for optimal transport as is used in the definition of Ollivier-Ricci curvature. We discuss preliminaries in optimal transport theory including the Monge and Kantorovich optimal tranport problem formulations which seek to determine the optimal cost associated with moving particles of mass from one distribution to another. We then introduce the fundamental Kantorovich-Rubinstein duality and finally the 1-Wasserstein distance and various properties, including proof of its form as a proper distance metric, and examples for transport on graphs. \\
\vspace{0.5cm}
\noindent
\textit{Chapter 2: Ollivier-Ricci Curvature on Metric Spaces}\\
This chapter provides an intensive exposition of Ollivier-Ricci curvature on general metric spaces,  with rigorous connection between the curvature properties of classical Ricci curvature as derived from the Riemann curvature tensor. Some of the results we discuss include variants of the $L^1$ and $L^2$ Bonnet-Myers theorems and proofs, as well as a discussion of our properties of the Levy-Gromov inequality hold for Ollivier-Ricci curvature. In this chapter we examine extensively the probability measures used to define the inputs of the 1-Wasserstein distance and ultimately the Ollivier-Ricci curvature, and furthermore, study the way in which random walks are deeply connected to this curvature notion on in both continuous and discrete settings.  
\\
\vspace{0.5cm}
\noindent
\textit{Chapter 3: Lin-Lu-Yau Extension to Undirected Graphs and Combinatorial Bounds}\\

This chapter is dedicated to the work of Yong Lin, Linyuan Lu, and Shing-Tung Yau in extending Ollivier-Ricci curvature to undirected graphs. We discuss various theoretical results presented by Lin-Lu-Yau, and their novel notion of Ricci curvature based off of Ollivier-Ricci, which is called Lin-Lu-Yau curvature. Finally, we discuss the work of J$\ddot{u}$rgen Jost and Shiping Liu in proving various combinatorial bounds for Ollivier-Ricci curvature which relies upon various combinatorial strategies to estimating transport cost between vertices and their neighbors. We will explore various proofs they present and their implications.\\
\vspace{0.5cm}
\noindent
\textit{Chapter 4: Considerations for Ollivier-Ricci Curvature of Directed Graphs}
Little research has been done for Ollivier-Ricci curvature on directed graphs. In this chapter, we seek to investigate this gap, and present some novel ideas and considerations for the extension of Ollivier-Ricci curvature to directed graphs, with a proposal for theoretical bounds involving Ollivier-Ricci curvature on edges within $k-$cycles, as well as considerations for out-branching and in-branching directed graph trees. In particular, we discuss a major challenge in this extension being the asymmetry of directed graphs and restriction of walks depending on specified directionalities. At the beginning of this chapter we introduce various essential background on directed graphs. \\
\vspace{0.5cm}
\noindent
\textit{Chapter 5: Curvature Informed Algorithms on Graphs and Networks}
The final chapter of this thesis explores how Ollivier-Ricci curvature for graphs has been utilized for various algorithms in applied research for network science and graph machine learning, including to application to the study of large networked data and graph neural networks. To this end, we introduce the primary two classes of graph algorithms in which curvature is applied, namely community detection or clustering, and graph rewiring which has applications in graph neural networks. \\
\vspace{0.5cm}
\textit{Appendix:} In the writing of this thesis, increasing amounts of background information were moved to the Appendix for reference so as to focus on the results of Ollivier-Ricci curvature. In this thesis we assume knowledge of basic Riemannian geometry, topology, graph theory, measure theory, theory of metric spaces, and various other advanced undergraduate or graduate-level concepts. However, the appendix provides various relevant exposition, definitions, and background information on these topics necessary for the understanding of the thesis for any curious reader.

\chapter{The 1-Wasserstein Distance}

We begin by discussing briefly the classical problem in probability theory of \textit{optimal transport}. French mathematician Gaspard Monge from the 16th century is credited with being the first to ask how to optimally move a particle from one mound to another, which was later formalized as being a problem of how to optimally transport a mass from one probability distribution to another. In order define to Ollivier-Ricci curvature upon which this thesis is based, it will be necessary for us to introduce fundamental results within the field of optimal transport theory. A major limitation of the formalized version of Monge's formulation is that mass cannot be split, such that each particle remains whole in the target distribution. Mathematician Leonid Kantorovich recognized this issue and introduced a new formulation in which mass splitting is possible, while still observing conservation of mass. A student of Kantorovich named Rubinstein contributed to this work on optimal transport theory, and this distance metric which will be heavily used in this paper for optimal transport for calculated transport cost, called the the Wasserstein Distance, is sometimes called the Kantorovich-Rubinstein distance. We discuss this background and prepare for the next chapter in which we discuss Ollivier-Ricci curvature and its dependence on the 1-Wasserstein distance.

\section*{Conventions}

We will use $C(X)$ to denote as the vector space of continuous functions $f: X \rightarrow \mathbb{R}$. We will denote the subspace of bounded continuous functions as $C_b(X)$.  For an open set $\Omega \subset \mathbb{R}^n$ and $k \in \mathbb{N}$, then $C^k(\Omega)$ is the set of functions with $k$ continuous derivatives in $\Omega$. Any function $f \in C^k(\overline{\Omega})$ denotes the restriction of a function in $C^k(\mathbb{R}^n)$ to $\overline{\Omega}$. $C_0(X)$ the space of continuous functions converging to $0$ at infinity. We will denote $C^\infty(\Omega)$ as the intersection of $C^k(\Omega)$ for $k \in \mathbb{N}$ and  $C^\infty_c(\Omega)$ for  compact support.

The space of Lipschitz functions will denoted as $Lip(X)$ and the subspace of bounded Lipschitz functions will be denotes as $Lip_b(X)$. The Lipschitz constant will be denoted $Lip(f)$ defined as follows: 
$$
Lip(f) := \sup_{x \neq y}\frac{|f(x) - f(y)}{|x-y|}
$$
A map $f$ between metric spaces is said to be $C$-Lipschitz if $d'(f(x), f(y)) \leq C d(x,y)$ for all $x, y$.



In a metric space $(X, d)$, $\mathcal{B}(X)$ denotes its Borel $\sigma$-algebra and $\mathcal{M}(X)$ the set of the $\sigma$-additive functions $\mu : \mathcal{B}(X) \to \mathbb{R}$. Furthermore, we denote by
\[
\mathcal{M}_+(X) := \{ \mu \in \mathcal{M}(X) : \mu \geq 0 \}, \quad \mathcal{P}(X) := \{ \mu \in \mathcal{M}_+(X) : \mu(X) = 1 \}
\]
the subsets of nonnegative and probability measures, respectively. We will use $\mathcal{L}^n$ to denote the Lebesgue measure in $\mathbb{R}^n$. $L^p$ is the Lebesgue space of exponent $p$.
The space of Borel probability measures on $X$ is denoted by $P(X)$. The weak topology on $P(X)$ is induced by convergence against $C_b(X)$, i.e., bounded continuous test functions. If $\mu \in P(X)$ and $\nu \in P(Y)$, then $\Pi(\mu, \nu)$ is the set of all joint probability measures on $X \times Y$ whose marginals are $\mu$ and $\nu$.
All measures of this chapter are Borel measures on Polish spaces.  We let $law(X)$ denote the law of a random variable $X$ defined on a probability space $(\Omega, \mathcal{P})$, equivalent to $X_{\#}\mathcal{P}$.

We define $W_p$ is the Wasserstein distance of order $p$ and $\delta_x$ is the Dirac mass at point $x$.

\begin{definition} (Polish space)
A topological space $(X, \tau )$ is Polish if it is
separable and completely metrizable.     
\end{definition}
\begin{definition} (Push Forward Operator and Measure) 
For a given Borel function $\phi: X \rightarrow Y$, the push forward operator $f_{\#}$ from $\mathcal{M}(X) \rightarrow \mathcal{M}(Y)$ is defined for all $B \in \mathcal{B}(Y)$ as
$$
f_{\#}\mu(B) := \mu(f_{-1}(B))
$$
where the measure $f_{\#}\mu$ is called the push forward measure
\end{definition}
Before we continue, we will discuss the inuition behind the push forward operator and measure. 
\begin{proposition}(Change of variables) \cite{lecturesOT} Given any Borel functions defined as follows:  
\begin{align*}
    f: X \rightarrow Y \\
    g: Y \rightarrow [0, \infty ]
\end{align*}
Then, 
$$
\int_Y gd f_{\#}\mu = \int_X (g \circ f)d\mu
$$
\end{proposition} 
Where therefore we have $g \circ f$ is $\mu$-integrable i.f.f. $\phi: X \rightarrow \overline{\mathbb{R}}$ is $f_{\#}\mu$-integrable.

\begin{definition} (Coupling of Measures) \cite{VillaniCedric2009Ot:o}
Let $(X, \mu)$ and $(Y, \nu)$ be probability spaces. Coupling the measures $\mu$ and $\nu$ is defined as constructing a measure $\pi$ on $X \times Y$ such that $\pi$ admits $\mu$ as a \textit{marginal} on $X$ and admits $\nu$ as a \textit{marginal} on $Y$.
\end{definition}
In other words, to couple measures $\mu$ and $\nu$, we construct random variables $X$ and $Y$ on some probability space $(\Omega, \mathcal{P})$ such that $X_{\#}\mathcal{P} = \mu$ and $Y_{\#}\mathcal{P} = \nu$. For example, the condition above is met when for all measurable sets $M_1 \subset X$ and $M_2 \subset Y$, we obtain $\pi[X \times M_2] = \nu[M_2]$ and $\pi[M_1 \times Y] = \mu[M_1]$. Equivalently the marginal condition is also met when for all $\pi[X \times M_2] = \nu[M_2]$ integrable measurable functions $f,g$ on $X,Y$ we get
$$
\int_X f d\mu + \int_Y g d\nu = \int_{X \times Y}(f(x) + g(y)) d\pi (x,y)
$$
\begin{definition}
Let $(X, \mu)$ and $(Y, \nu)$ be probability spaces. A coupling $(X,Y)$ is a \textit{deterministic coupling} if there exists a measurable function $T$ such that $T: X \rightarrow Y$ such that $Y = T(X)$. By definition, $(X,Y)$ being a deterministic coupling is an equivalent statement to $(Id, T)_{\#}\mu = \pi$, and is also equivalent to achieving the following when $X_{\#}\mathcal{P}$ and $Y = T(X)$: 
\begin{enumerate}
    \item $T_{\#} \mu = \nu$ 
    \item $\int_X f(T(x))d\mu(x)  = \int_Y f(y)d\nu(y)$
\end{enumerate}
for $\nu$-integrable functions $f$.

We will be referring to the map $T$ going forward as a \textit{transport map} as classically done in optimal transport theory. We will be using $T$. Mass can be represented by a measure $\mu$. We say $T$ transports mass from $\mu$ to $\nu$, and therefore $T$ is a map transporting mass map between measures. 
\begin{theorem} (Existence of optimal coupling \cite{VillaniCedric2009Ot:o})
Let $(X, \mu)$ and $(Y, \nu)$ be distinct Polish probability spaces. Define
$f : X \to \mathbb{R} \cup \{-\infty\}$ and
$g : Y \to \mathbb{R} \cup \{-\infty\}$ as upper semicontinuous functions whereby $f \in L^1(\mu)$, $g \in L^1(\nu)$.

Define $c : \mathcal{X} \times \mathcal{Y} \to \mathbb{R} \cup \{+\infty\}$ 
as a lower semicontinuous cost function, such that 
$$c(x,y) \geq f(x) + g(y), \forall x, y$$ Then there exists a coupling of $(\mu, \nu)$ that minimizes the total cost  $c(X,Y)$ for all possible couplings $(X,Y)$.
\end{theorem}
\begin{theorem} (Optimal cost convexity \cite{VillaniCedric2009Ot:o})
Let $(X, \mu)$ and $(Y, \nu)$ be distinct Polish probability spaces. Define
$c : X \times Y \rightarrow \mathbb{R} \cup +\infty\}$ to be a lower semicontinuous function. On $\mathcal{P}(X) \times \mathcal{P}(Y)$ let the optimal transport cost functional be $C$.

Define $c : \mathcal{X} \times \mathcal{Y} \to \mathbb{R} \cup \{+\infty\}$ 
as a lower semicontinuous cost function, such that 
$$c(x,y) \geq f(x) + g(y), \forall x, y$$ Then there exists a coupling of $(\mu, \nu)$ that minimizes the total cost  $c(X,Y)$ for all possible couplings $(X,Y)$.
\end{theorem}
\begin{proposition}
Optimal couplings always exist by Theorem 2.2.6. However, deterministic couplings do not always exist. \cite{VillaniCedric2009Ot:o}
\end{proposition}
As such, transport maps do not always exist. We will formally define transport maps in the next section and relate it to Monge's formulation for the optimal transport problem.
\end{definition}
\begin{definition}
The map $T: X \rightarrow Y$ is defined as a \textit{transport map} which transports $\mu \in \mathcal{P}(X)$ to $\nu \in \mathcal{P}(Y)$ if $\nu(B) = \mu(T^{-1} B))$ for all $\nu$-measurable sets B. 
\end{definition}
Again, recall that if $\nu(B) = \mu(T^{-1} B))$ for all $\nu$-measurable sets $B$, then we can write $\nu = T_{\#}\mu$. By Proposition 2.28, for any given probability measure $\mu$ and $\nu$ the transport map $T$ such that $\nu = T_{\#}\mu$ may not exist.

\begin{figure}[htbp]
    \centering
    \makebox[\textwidth]{
        \begin{minipage}{0.5\textwidth}
            \centering
            \includegraphics[width=\textwidth,keepaspectratio]{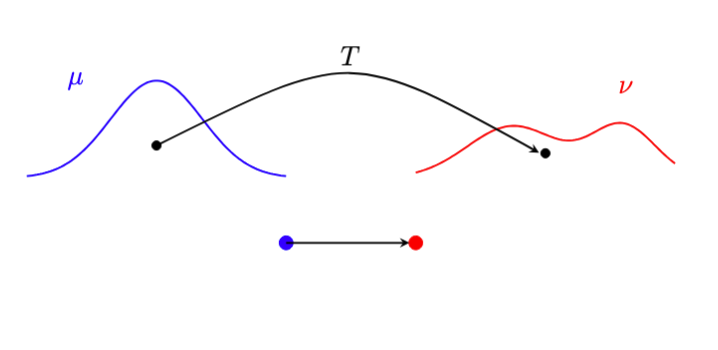}
            \subcaption{Monge Formulation}
            \label{fig:left_graph}
        \end{minipage}
        \hfill
        \begin{minipage}{0.5\textwidth}
            \centering
            \includegraphics[width=\textwidth,keepaspectratio]{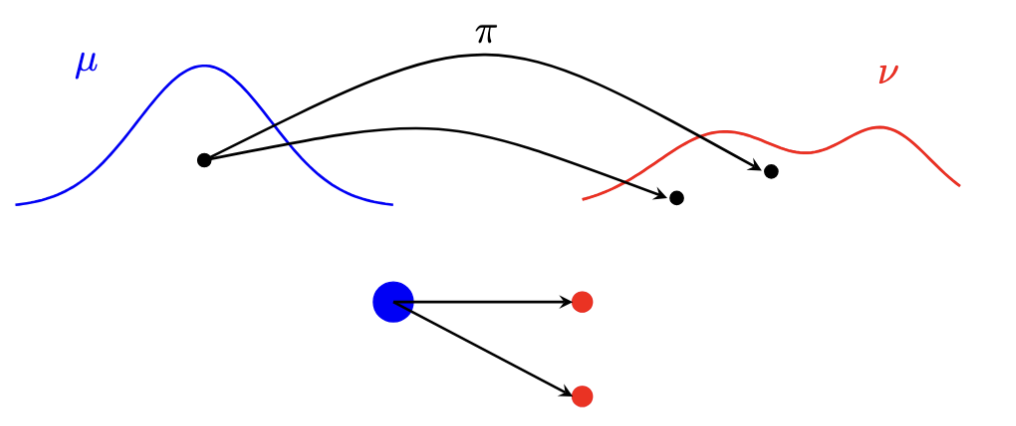}
            \subcaption{Kantorovich Formulation}
            \label{fig:right_graph}
        \end{minipage}
    }
    \caption{Mass transport in Monge Formulation vs. Kantorovich Relaxation}
    \label{fig:side_by_side_graphs}
\end{figure}
\begin{figure}
    \centering
    \includegraphics[width=0.5\linewidth]{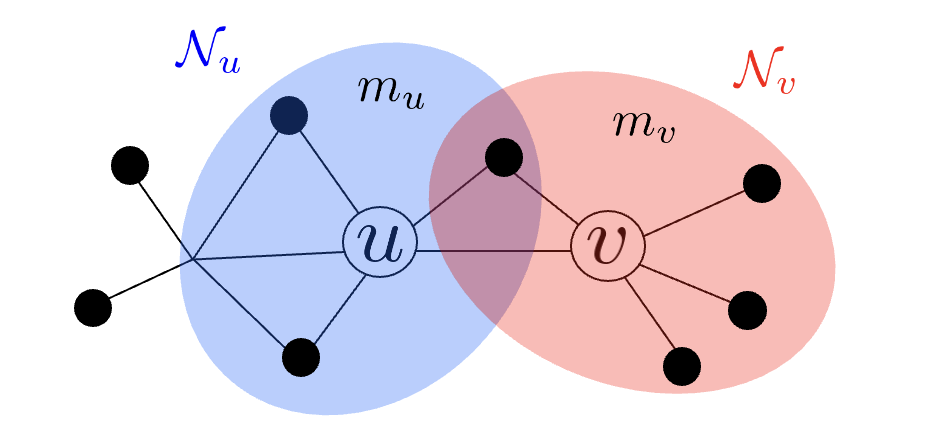}
    \caption{Neighborhood of $u$ in green, neighborhood of $v$ in blue. Each neighboor is endowed with measures $m_u$ and $m_v$ respectively for a graph $G$.}
    \label{fig:enter-label}
\end{figure}

\section*{Optimal Transport Problem Formulations}
The problem of optimal transport as proposed by Monge is the problem of finding a way to minimize the transportation cost of moving mass from one distribution of mass to another. His problem was later formalized in terms of probability theory where we refer to a given mass by a measure like $\mu$ and its mass distribution by a probability distribution $\mathcal{P}(X)$ for some random variable $X$. Then for $\mu \in \mathcal{P}(X)$ and $\nu \in \mathcal{P}(Y)$, a Borel cost function represents the cost of moving a unit of mass from point $x$ to a point $y$ defined as
$$
c(x,y): X \times Y \rightarrow [0, \infty]
$$
We will formalize this below. 
\begin{definition} \cite{VillaniCedric2009Ot:o} (Monge Optimal Transport Problem):
If we are given probability measures $\mu \in \mathcal{P}(X), \nu \in \mathcal{P}(Y)$ then Monge's problem is formulated as follows: 
$$
inf \int_X c(x,T(x))d\mu(x)
$$
where 
over all Borel $\mu$-measurable transport maps $T: X \rightarrow Y$ for which $\nu = T_{\#}\mu$.
\end{definition}
The cost function $c$ defined in this problem was originally defined as $L^1$ cost where $(c,x) = |x-y|$ and is more difficult to solve with $L^2$ cost where $(c,x) = |x-y|^2$. 
We will now interpret the Monge formulation in terms of mass transfer in the discrete setting. In this formulation, mass is not split as we are defining maps $x \mapsto  T(x)$. If we are given $X = \{x_1, \dots, x_m\}$, with a probability distribution $\mu = (\mu_1, \dots, \mu_m)$, and set of \textit{target locations} $Y = \{y_1, \dots, y_n\}$, with a probability distribution $\nu = (\nu_1, \dots, \nu_n)$ we define our cost function as
$$
c: X \times Y \to \mathbb{R}^{\geq 0} 
$$
where $ c(x_i, y_j) $ represents the cost of moving mass from $ x_i $ to $ y_j $. The Monge problem is then formulated in the discrete case as:
\[
\min_{T} \sum_{x_i \in X} c(x_i, T(x_i)) \mu_i.
\]
\begin{example}
Define $X = \{x_1,...,x_k \}$ and $Y = \{y_1,...,y_k\}$ such that $|X|=|Y|=k$. Define $\delta$ as the Dirac measure or Dirac delta function in which $\delta_{x_i}$ places all mass at $x_i$. We can define a probability measure on $X$ as $\mu = \sum_{i = 0}^k \frac{1}{k} \delta_{x_i}$. Similarly we can define a probability measure on $Y$ as $\nu = \sum_{j = 0}^k \frac{1}{k} \delta_{y_i}$. Then we define our map $T: X \rightarrow Y$ where $T_{\#}\mu = \nu = \sum_{j = 0}^k \frac{1}{k} \delta_{y_i}$.
\end{example}

We now move on to Kantorovich's relaxation, also called Kantorovich's problem formulation. As previously discussed, it allows more flexibility in finding optimal transport plans because mass is allowed to be split and mapped to a variety of different locations in the target probability distribution. For example, on a graph such as that pictured in Figure 1.2. with vertices $u,v$ connected by an edge with common and non-common neighbors, then Kantorovich's relaxation of the optimal transport problem would allow mass assigned to a neighbor of $u$ to be split such that portions of this mass are mapped distributed to the three neighbors of $v$, so long as no mass is created or destroyed. We formalize this below. Furthermore, an essential note to make here is that we no longer use the term \textit{transport map} for this relaxation, but rather a \textit{transport plan} which is a matrix that assigns mass at a location on a starting probability distribution and systematically maps it to any number of possible split locations in the target probability distribution. Again, the cost of such transport is captured my the cost functions being applied to each of these distributions of mass. 

\begin{definition} (Kantorovich's Optimal Transport Problem) \cite{VillaniCedric2009Ot:o}:
If we are given probability measures $\mu \in \mathcal{P}(X), \nu \in \mathcal{P}(Y)$ then Kantorovich's problem is formulated as follows: 
$$
\inf \int_{X \times Y} c(x,y)d\pi(x,y)
$$
over the set of all transport plans between $\mu$ and $\nu$, i.e. over all $\pi \in \Pi(\mu, \nu)$.
\end{definition}

In the discrete case, suppose we are given $X = \{x_1, \dots, x_m\}$, with a probability distribution $\mu = (\mu_1, \dots, \mu_m)$, and set of \textit{target locations} $Y = \{y_1, \dots, y_n\}$, with a probability distribution $\nu = (\nu_1, \dots, \nu_n)$ and where $ c: X \times Y \to \mathbb{R}^{\geq 0} $, where $ c(x_i, y_j) $ represents the cost of moving mass from $ x_i $ to $ y_j $.

The Kantorvich problem involves finding a \textit{transport plan} matrix $ \pi \in \mathbb{R}^{m \times n}_{\geq 0} $ that satisfies:
\[ \sum_{j} \pi_{ij} = \mu_i, \quad \sum_{i} \pi_{ij} = \nu_j. \]
This allows mass from $x_i$ to be split among multiple $y_j$.

The \textit{Kantorovich problem} is then formulated as:
\[ \min_{\pi \geq 0} \sum_{i,j} c(x_i, y_j) \pi_{ij}. \]

The Monge and Kantorovich problem formulations simply formalize the problem of mass transport across distributions. The problem remains: how can we solve this problem and how can we do so in a way that is optimal? Our goal is to minimize the average cost of transporting mass distributed according to $\mu$ and $\nu$. The contribution of the Kantorovich-Rubenstein duality is the result that this minimization problem can actually be solved by maximizing for the difference in expected cost of any 1-Lipschitz function $f$ over the distributions $\mu$ and $\nu$.

\begin{theorem} (Kantorovich-Rubenstein Theorem) \cite{lecturesOT}
Let X and Y be Polish spaces with $\mu \in \mathcal{P}(X)$ and $\nu \in \mathcal{P}(Y)$. If the space $(X,d)$ is compact, then: 
$$
\inf_{\pi \in \mathcal{P(\mu, \nu)}} \int_{X \times X} d(x,y) \pi(dx dy) = \sup \{ \int_X f d\mu - \int_X f d\nu : \left \| f \right \|_L \leq 1\}
$$
where $\mu(A) = \pi(A \times X)$, $\nu(A) = \pi(X \times A)$ for all $A \in A \in \mathcal{B}(X)$, the set of all Borel measures on $X$. $\mathcal{P}(\mu, \nu)$ is the set of regular probability Borel measures $\pi$ on the topological product $X \times X$. We define $\left \| f \right \|_L$ for $f: X \rightarrow \mathbb{R}$ as follows: 
$$
\left \| f \right \|_L = sup\{ \frac{|f(x) - f(y)|}{d(x,y)} \mid x,y \in X \mid  x \neq y\}
$$
\end{theorem}
for $1-$Lipchitz functions $f$. 

If we are seeking to find the optimal transport cost of moving mass from some source of a given distribution and optimally distributing this mass to a different distribution, then recall that if $c(x,y)$ is the cost of transporting a single unit of mass from $x$ to $y$ that we can model this optimal transport cost between measures $\mu$ and $\nu$ as follows: 

\begin{align*}
    \inf_{\pi \in \Pi (\mu, \nu)} \int c(x,y) d\pi (x,y)
\end{align*}

A classic idea in optimal transport is to consider this optimal cost as a \textit{distance}  between measures $\mu$ and $\nu$, but in an adjusted form that would satisfy all of the required distance axioms. This leads us to define the Wasserstein distance as follows. 

\section*{p-Wasserstein Distance}

\begin{definition}
    Let $(\mathcal{X}, d)$ be a Polish metric space, and let $p \in [1, \infty)$. For any two probability measures $\mu, \nu$ on $\mathcal{X}$, the Wasserstein distance of order $p$ between $\mu$ and $\nu$ is defined by the formula
\begin{equation}
W_p(\mu, \nu) = \left( \inf_{\pi \in \Pi(\mu, \nu)} \int_{\mathcal{X}} d(x, y)^p \, d\pi(x, y) \right)^{\frac{1}{p}} 
\end{equation}
where $W_p(\mu, \nu)$ satisfies the axioms of distance for any $p \in [1, \infty)$. Let $\mu, \nu, \gamma$ be probability measures on $\mathcal{X}$. Then the p-Wasserstein satisfies (1) nonnegativity, (2), symmetry, (3) the triangle inequality (4) distance between a point and itself is zero:
\begin{enumerate}
    \item  $W_p(\mu, \nu) \geq 0$
    \item  $W_p(\mu, \nu) = W_p(\nu, \mu)$
    \item  $W_p(\mu, \gamma) \leq W_p(\nu, \mu) + W_p(\mu, \gamma)$
    \item $W_p(\mu, \mu) = 0$
\end{enumerate}

\end{definition}

We will work with $p = 1$ for the remainder of the thesis due to Ollivier's definition. Furthermore, note that we can then use the Kantorovich-Rubenstein duality to turn this into a supremum. We will define Wasserstein distance on graphs in Chapters 3 and 4, and in particular, we will demonstrate how to compute optimal transport costs for small cases. For example, see computations in examples \ref{example1} through \ref{example2}. Observe that the 1-Wasserstein distance can be computationally expensive, and due to the complexity of any optimization problem can be difficult to do by hand. Various algorithms are often used for the computation of the 1-Wasserstein distance, including the Hungarian and Sinkhorn algorithms which we reference in the appendix.

\chapter{Ollivier-Ricci Curvature on Metric Spaces}

This chapter is an exposition dedicated to rigorously introducing the Ricci curvature notion introduced by Yann Ollivier from 2007-2009 \cite{ollivier_ricci_2007},  \cite{ollivier_ricci_2009} for arbitrary metric spaces, which has been highly influential the field of discrete differential geometry and was was further studied by mathematicians such as Shing Tung Yau with application to graphs and theoretical computer scientists for application in complex networks and machine learning. In his 2009 paper, Ollivier introduces a notion of Ricci curvature of metric spaces and extends this definition to any Markov chain on a metric space, though we will focus more generally on how this relates the transport from one measure to another as discussed in Chapter 1. This field has coined his notion of Ricci curvature as Ollivier Ricci Curvature or Ollivier's Ricci Curvature. We will use the former. This notion of Ricci curvature is defined in terms of how much small balls are closer than their centers are in terms of 1-Wasserstein distance, and can be used to define Ricci curvature on a Riemannian without requiring derivatives, along with other advantages which we will discuss. Furthermore, it natural extends to defining discrete Ricci curvature on graphs. This chapter will first introduce Ricci curvature as defined by Gregorio Ricci-Curbastro for defining curvature of Riemannian manifolds equipped with Levi-Civita connections and in terms of the Riemannian curvature tensor. We will then introduce Ollivier-Ricci curvature on Riemannian manifolds, and then extensions of this work to graphs. Nearly the entire body of work on Ollivier-Ricci curvature extension to graphs involves undirected graphs. We will then introduce the current body of work studying Ollivier-Ricci curvature on directed graphs. 

Yann Ollivier's original paper proposed a Wasserstein-distance based Ricci curvature notion based on metric spaces. He originally defined this notion with respect to random walks on a metric space, referring simply to a family of measures. Ollivier-Ricci Curvature can be applied to both continuous and discrete settings depending upon the discretization of the Wasserstein transportation distance used. It can be applied to the classical Riemannian case of Ricci curvature and is a desirable method of calculation of Ricci curvature on manifolds that does not require the use of derivatives. On graphs, it is desirable for existing established methods of calculating Wasserstein distance. 

Now we will define Ollivier-Ricci Curvature in terms of the $W_1$ distance introduced in Chapter 1, and the distance $d(x,y)$ between points $x$ and $y$. 
\\
\begin{definition} (Ollivier-Ricci Curvature \cite{ollivier_ricci_2009}) Let $(X,d)$ be a metric space endowed with a random walk $m$. Let $x,y \in Y$ 
and $x\neq y$. The \textit{Ollivier-Ricci curvature} of $(X,d,m)$ along $(xy)$ is 
\begin{equation}
\operatorname{Ric_O}(x,y) := 1 - \frac{W_1(m_x,m_y)}{d(x,y)}
\end{equation}
\end{definition}
We first provide some examples. 
In the case of undirected graphs, we can define Ollivier-Ricci curvature as follows: 
Let $G(V,E)$ be an unweighted undirected graph where $d(x,y)$ denotes the distance in edge count between vertices $x$ and $y$ and $W_1$ is defined as follows. Then, 
the \textit{Ollivier-Ricci curvature} of of an edge $xy$ in $E(G)$ is defined as: 
\begin{equation}
\operatorname{Ric_O}(x,y) := 1 - \frac{W_1(m_x,m_y)}{d(x,y)}
\end{equation}

Ollivier's derivation of this notion comes from the anology between probability measures $m_x, m_y$ and spheres $S_x, S_y$ of points $x,y$ where the relationship between the Wasserstein distance $M_1(m_x,m_y)$ and $d(x,y)$ distance can be used to assess whether a curve or an edge for example are positive or negative curvature. It is synthetically designed such that the following relations hold: 
\begin{itemize}
    \item $\operatorname{Ric_O}(x,y) > 0$  when $W_1(m_x,m_y) < d(x,y)$
    \item $\operatorname{Ric_O}(x,y) < 0$  when $W_1(m_x,m_y) > d(x,y)$
    \item $\operatorname{Ric_O}(x,y) = 0$ when $W_1(m_x,m_y) = d(x,y)$
\end{itemize}
In \cite{ollivier_ricci_2009}, essential properties of Ricci curvature on Riemannian manifolds are shown to be conserved or have analogous results for Ollivier-Ricci curvature on metric spaces. Before I introduce and prove selected results for Ollivier-Ricci curvature, I review essential properties of Ricci curvature on Riemannian manifolds. Many of the results Ollivier introduces are arguments that the Ollivier-Ricci notion retains key results  of classical Ricci curvature on Riemannian manifolds, especially control of geodesic divergence and volume growth. We begin this discussion by reviewing classical Ricci curvature. In particular, we will discuss the Bonnet-Myers Theorem and the  Levy-Gromov isoperimetric inequality, and their analogous extensions for Ollivier's Ricci curvature for general metric spaces. 
\begin{definition}
Let $(M^n,g)$ be a Riemannian manifold equipped with the Levi-Civita connection. The \textbf{Ricci curvature \textit{Ric}} of $M$ is defined as follows: For every $p \in M$, for all $x, y \in T_{p} M$, set $\operatorname{Ric}_{p}(x, y)$ to be the trace of the endomorphism $v \mapsto R_{p}(x, v) y$. With respect to any orthonormal basis $\left(e_{1}, \ldots, e_{n}\right)$ of $T_{p} M$, we have

$$
\operatorname{Ric}_{p}(x, y)=\sum_{j=1}^{n}\left\langle R_{p}\left(x, e_{j}\right) y, e_{j}\right\rangle_{p}=\sum_{j=1}^{n} R_{p}\left(x, e_{j}, y, e_{j}\right)
$$
\end{definition}

The scalar curvature $S$ of $M$ is the trace of the Ricci curvature; that is, for every $p \in M$,

$$
S(p)=\sum_{i \neq j} R\left(e_{i}, e_{j}, e_{i}, e_{j}\right)=\sum_{i \neq j} K\left(e_{i}, e_{j}\right)
$$

where $K\left(e_{i}, e_{j}\right)$ denotes the sectional curvature of the plane spanned by $e_{i}, e_{j}$.

In a chart the Ricci curvature is given by

$$
R_{i j}=\operatorname{Ric}\left(\frac{\partial}{\partial x_{i}}, \frac{\partial}{\partial x_{j}}\right)=\sum_{m} R_{i j m}^{m}
$$

and the scalar curvature is given by

$$
S(p)=\sum_{i, j} g^{i j} R_{i j}
$$

where $\left(g^{i j}\right)$ is the inverse of the Riemann metric matrix $\left(g_{i j}\right)$. 

The Ricci tensor is a $(0,2)$-tensor. However it can be interpreted as a $(1,1)$-tensor given the next definition.

\begin{definition}
For Riemannian manifold $(M^n, g)$,
$$
\left\langle\operatorname{Ric}_{p}^{\#} u, v\right\rangle_{p}=\operatorname{Ric}_{p}(u, v)
$$

for all $u, v \in T_{p} M$.
\end{definition}

\begin{proposition}
Let $(M^n, g)$ be a Riemannian manifold and let $\nabla$ be any connection on $M$. If $\left(e_{1}, \ldots, e_{n}\right)$ is any orthonormal basis of $T_{p} M$, we have

$$
\operatorname{Ric}_{p}(u)=\sum_{j=1}^{n} g^{jk} R_{p}\left(e_{j}, u\right) e_{j}
$$

Then, 

$$
S(p)=\operatorname{tr}\left(\operatorname{Ric}_{p}^{\#}\right)
$$
\end{proposition}

For any orthonormal frame $\left(e_{1}, \ldots, e_{n}\right)$ of $T_{p} M$, by the definition of the sectional curvature $K$, we obtain:

$$
\operatorname{Ric}\left(e_{1}, e_{1}\right)=\sum_{i=1}^{n}\left\langle\left(R\left(e_{1}, e_{i}\right) e_{1}, e_{i}\right\rangle=\sum_{i=2}^{n} K\left(e_{1}, e_{i}\right)\right \rangle
$$

Thus, $\operatorname{Ric}\left(e_{1}, e_{1}\right)$ is the sum of the sectional curvatures of any $n-1$ orthogonal planes orthogonal to $e_{1}$ (a unit vector).

For a Riemannian manifold with constant sectional curvature, we have

$$
\operatorname{Ric}(x, x)=(n-1) K g(x, x), \quad S=n(n-1) K
$$

where $g=\langle-,-\rangle$ is the metric on $M$.\\
\vspace{0.5in}
We will now discuss some essential results regarding curvature and Ricci curvature on manifolds that will be relevant to Ollivier's notion of curvature.

\begin{definition}
The diameter a manifold is defined as:
$$
\operatorname{diam}(M^n):=\sup \{d(p, q): p, q \in M\}
$$
\end{definition}
\begin{theorem} (Bonnet-Myers Theorem)
Suppose $\left(M^{n}, g\right)$ is complete with $\operatorname{Ric}(g) \geq \frac{n-1}{r} g$, then:
$$
\operatorname{diam}\left(M^{n}, g\right) \leq \pi r
$$
Note that we define $\operatorname{Ric}(g) \geq \alpha g$ if for every vector field $X$,
$$
\operatorname{Ric}(g)(X, X) \geq \alpha|X|^{2}
$$
\end{theorem}
We will introduce now the classical isoperimetric inequality in order to discuss the Levy-Gromov isoperimetric inequality in order to relate Ricci curvature to volume.  If we denote $S^{n+1}$ to an $(n+1)$-dimensional sphere let $V_{0} \subset S^{n+1}$ be a domain with smooth boundary $\partial V_{0} $, then the classical inequality gives us a lower bound on the $n-$dimensional volume $\operatorname{Vol}(\partial V_0)$ of the boundary $\partial V_0$. We denote $\operatorname{Vol}(V_0)$ as the $(n+1)-$dimensional volume of the domain $V_0$. We denote the following ratio as $\alpha$:
$$
\alpha = \frac{\operatorname{Vol}(V_0)}{\operatorname{Vol}(S^{n+1})}
$$
For a ball $B_{\alpha} \subset S^{n+1}$, we define its volume:
$$
\operatorname{Vol}(B_{\alpha}) = \alpha \operatorname{Vol} (S^{n+1}).
$$
For the boundary sphere $\partial B_\alpha$ we will denote $s(\alpha)$ to be its $n-$dimensional volume as done by, \cite{GromovMikhail2001CPLI}. Furthermore, as is classically done we denote the ratio between $s(\alpha)$ and $\operatorname{Vol}(S^{n+1})$ to be $I_{S_{n+1}}(\alpha)$ as follows:
$$
I_{S_{n+1}}(\alpha) = \frac{s(\alpha)}{\operatorname{Vol}(S^{n+1})}
$$
\begin{theorem}\label{thm:isoperi} \cite{osserman1978isoperimetric}\cite{GromovMikhail2001CPLI}
Equipped with the notation above, the classical isometric inequality is as follows: 
$$
\frac{\operatorname{Vol}(\partial V_0)}{\operatorname{Vol}(S^{n+1})} \leq I_{S_{n+1}}(\alpha)
$$
\end{theorem}
\begin{example}
The classical inequality is usually formulated with regards to the perimeter of a closed curve in a plane and the corresponding enclosed area. Let $A$ be the area enclosed by a closed curve of length $L$. Then we have $4\pi A \leq L^2$ according to this classical inequality, where $4\pi A = L^2$ is achieved if and only if the curve is a circle. See \cite{osserman1978isoperimetric} for a proof of this classical case and extensions. The inequality allows one to answer the following problem formulation: if perimeter is fixed, then for all closed curves in a given plane, which curve maximizes its enclosed area?
\end{example}
\begin{example} More generally if we consider $R^n$, then the isoperimetric inequality reveals that for all $X \subset R^n$, we obtain $| \partial X| \geq |\partial B|$
where we define $B$ as the ball satisfying $|B| = |X|$. 
\end{example}
\begin{example} Similarly consider the $n-$dimensional sphere $\mathbb{S}^n$. By the isoperimetric inequality we have that for all $X \subset \mathbb{S}^n$ we have that $|\partial X| \geq |\partial B|$ for metric ball $B$ representing a spherical cap satisfying $|B| = |X|$. 
    
\end{example}

We can now move on to the Levy-Gromov isoperimetric inequality, introduced by Gromov as an extension of Levy's isoperimetric inequality for convex hypersurfaces in $\mathbb{R}^{n+1}$. Gromov extends this to all Riemannian manifolds.
\\
\begin{theorem} \label{thm:levy-gromov}
\cite{GromovMikhail2001CPLI}
(Levy-Gromov Isoperimetric inequality)\\
For a compact Riemannian manifold $(M, g)$, we denote it's total volume as done by \cite{GromovMikhail2001CPLI} as $Vol(M)$. Let $M$ be a closed compact $(n+1)$-dimensional manifold and let $E \subset M$ be a domain with smooth boundary $\partial E$. \\
Define $T(M) = inf _{\tau} \operatorname{Ric}(\tau, \tau)$ for $\tau$ over all unit tangent vectors of $M$. If we have:
$$
T(M) \geq n = T(S^{n+1})
$$
Then, 
\begin{equation*}
\frac{\operatorname{vol}\left(\partial E\right)}{\operatorname{vol}(M)} \geq \mathrm{Is}_{n+1}(\alpha), \quad \alpha=\frac{\operatorname{vol}\left(E\right)}{\operatorname{vol}(M)} 
\end{equation*}

where $\operatorname{Is}_{n+1}(\alpha)$ is the equivalent function as in the classical inequality from Theorem~\ref{thm:isoperi}.
\end{theorem}
\begin{remark}
The above notation is mostly aligned with the original notation and interpretation from Gromov's original paper. Equivalently we could say as follows. Let Riemannian manifold $(M^n, g)$ with $\operatorname{Ric}_g \geq (n-1)g$ and $E, \partial E$ defined the same as above. Furthermore let $\mathbb{S}^n$ denote the sphere of unit radius and $B \subset \mathbb{S}^n$ be a metric ball satisfying $\frac{|E|}{|M|} = \frac{|B|}{|\mathbb{S}^n|}$, then: 
$$
\frac{|\partial E|}{|M|} \geq \frac{|\partial B|}{|\mathbb{S}^n|}
$$
\end{remark}
\begin{remark}
The Levy-Gromov theorem shows us that among Riemannian manifolds with a fixed positive lower bound on the Ricci tensor, round spheres have the most constrained way of enclosing a given amount of volume relative to their total size. This result shows that a lower bound on Ricci curvature imposes strict limitations on how efficiently volume can be enclosed within a space. In other words, Among all compact manifolds with a given lower Ricci curvature bound, the unit sphere $S^{n+1}$ has the smallest possible surface area for enclosing a given volume fraction. In other words, 
if volume is enclosed within a domain 
$E$ in a manifold respecting the proeprties of the theorem, the surface area of boundary $\partial E$ cannot be smaller than that of the corresponding spherical cap in $S^{n+1}$. 
\end{remark}

\section*{Ollivier-Ricci Curvature Results}

Recall as introduced at the beginning of this chapter that Ollivier-Ricci curvature is defined as follows: 
\begin{definition} (Random walk) Let $(X,d)$ be a Polish metric space endowed with its Borel $\sigma-$algebra. We define a \textit{random walk} $m$ on the space $X$ as the family of probability measures $m_x(\cdot)$ on $X$ for all $x \in X$. The following assumptions are satisfied by $m$: 
\begin{enumerate}
    \item Each measure $m_x$ for $x \in X$ depends measurably on its respective point $x$. 
    \item Each measure $m_x$ has finite first moment, meaning that for any $z\in X$ and $x \in X$ we have: 
$$
\int d(z,y) d m_x (y) < \infty
$$
\end{enumerate} 
\end{definition}
\begin{example}
Note that many examples we use will use what is called a \textit{lazy} random walk. This is simply a random walk $m$ in which the probability of staying in the current state is non-zero, in other words, introducing a lazy option to remain stationary. Specifically, $m_x(x) = \frac{1}{2}$  and for possible sequential step to $y$, then $m_x(y) = \frac{1}{2n}$ if there are $n$ other possible non-stationary steps. 
\end{example}
\begin{definition} 
Let $(X,d)$ be a metric space endowed with a random walk $m$. Let $x,y \in Y$ 
and $x\neq y$. The \textit{Ollivier-Ricci curvature} of $(X,d,m)$ along $(xy)$ is 
\begin{equation}
\operatorname{Ric_O}(x,y) := 1 - \frac{W_1(m_x,m_y)}{d(x,y)}
\end{equation}
\end{definition}
\begin{remark}
Note that Ollivier refers to this as $\kappa(x,y)$ but for clarity and attribution purposes we have denoted this notion with a subscript O. The rest of the notion remains the same. Additionally in \cite{ollivier_ricci_2009}, this is referred to as coarse Ricci curvature by Ollivier but has since been named Ollivier-Ricci or some similar variation of this. 
\end{remark}
\begin{remark} \label{thm: analogy}(The $m_x$ to $S_x$ analogy) The most essential takeaway here is that that measures $m_x, m_y$, for example, are analogous to having two close spheres $S_x, S_y$ in the classical Ricci notion. We can characterize positive Ricci curvature on a Riemannian manifold with regards to spheres around two close points $x$ and $y$, where if there are close spheres $S_x, S_y$ which are closer together by some notion of distance compared with their center points $x,y$ this informally can be thought of as representing positive Ricci curvature. To be more explicit, for closeby points $x,y$ on some Riemannian manifold where we can define a tangent vector by $(xy)$ and additional tangent vectors $v$  at $x$ and $v'$ at $y$ (where $v'$ is found via parallel transport of $v$ from point $x$ to $y$. In positive Ricci curvature we observe that the geodesic from point $x$ and following vector $v$ with get closer together with the geodesic from point $y$ following vector $v'$. In negative Ricci curvature the geodesics get farther apart than the distance between the center points, and in zero Ricci curvature its easy to see that if there is flatness, then there the distance between the geodesics as they leave $S_x, S_y$ with remain the same as $d(x,y)$. In reality, this is really sectional curvature, and the idea of Ricci curvature is that this occurs on average averaging when all tangent vectors at $x,y$. This general concept is visualized in Figure \ref{fig:spheres}. In Ollivier-Ricci curvature, the the differences between the transportation cost or 1-Wasserstein distance between measures $m_x, m_y$ compared with the actual distance between $x,y$ follows the same analogy as the geodesic dispersion and distance compared with distance between points $x, y$ as previously discussed. This is visualized in \ref{fig:measures} where recall that measures are mapping the points $x,y$ to a real-valued number between $0$ and $1$. Of course geometrically, this is not easily visualized like spheres of the previous figure, but is indented to represent the analogy. 
\end{remark}
\begin{figure}
    \centering
    \includegraphics[width=0.8\linewidth]{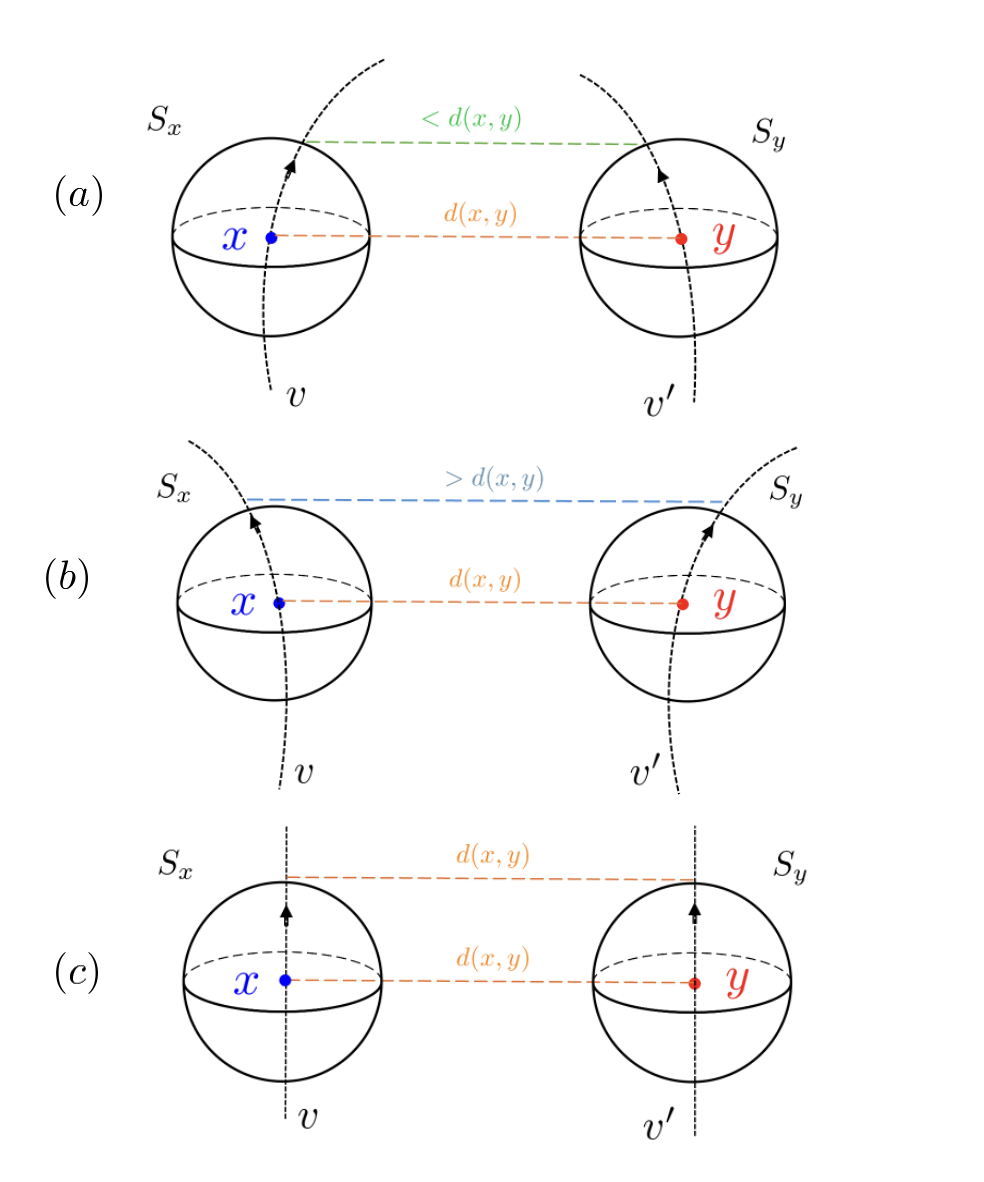}
    \caption{Remark \ref{thm: analogy} illustration: (a) positive Ricci curvature, (b) negative Ricci curvature,  (c) zero Ricci curvature.}
    \label{fig:spheres}
\end{figure}
\begin{definition} ($\epsilon$-random walk). Let $(X, d, \mu)$ be a metric space endowed with the measure $\mu$. Let balls in $X$ have finite measure. ALso let the support of $\mu$ be equal to $X$, denoted $\operatorname{Supp} \mu = X$. For some $\epsilon > 0$, we define the $\epsilon$-\textit{step random walk}, $m_x$ on $X$ as a random walk which starts at point $x$ and which randomly jumps in the ball of radius $\epsilon$ around center point $x$. Each of these jumps has probability proportional to $\mu$ as described below: 
$$
m_x = \mu_{\mid (B(x, \epsilon))}
$$ 
\end{definition}
\begin{figure}
    \centering
    \includegraphics[width=0.7\linewidth]{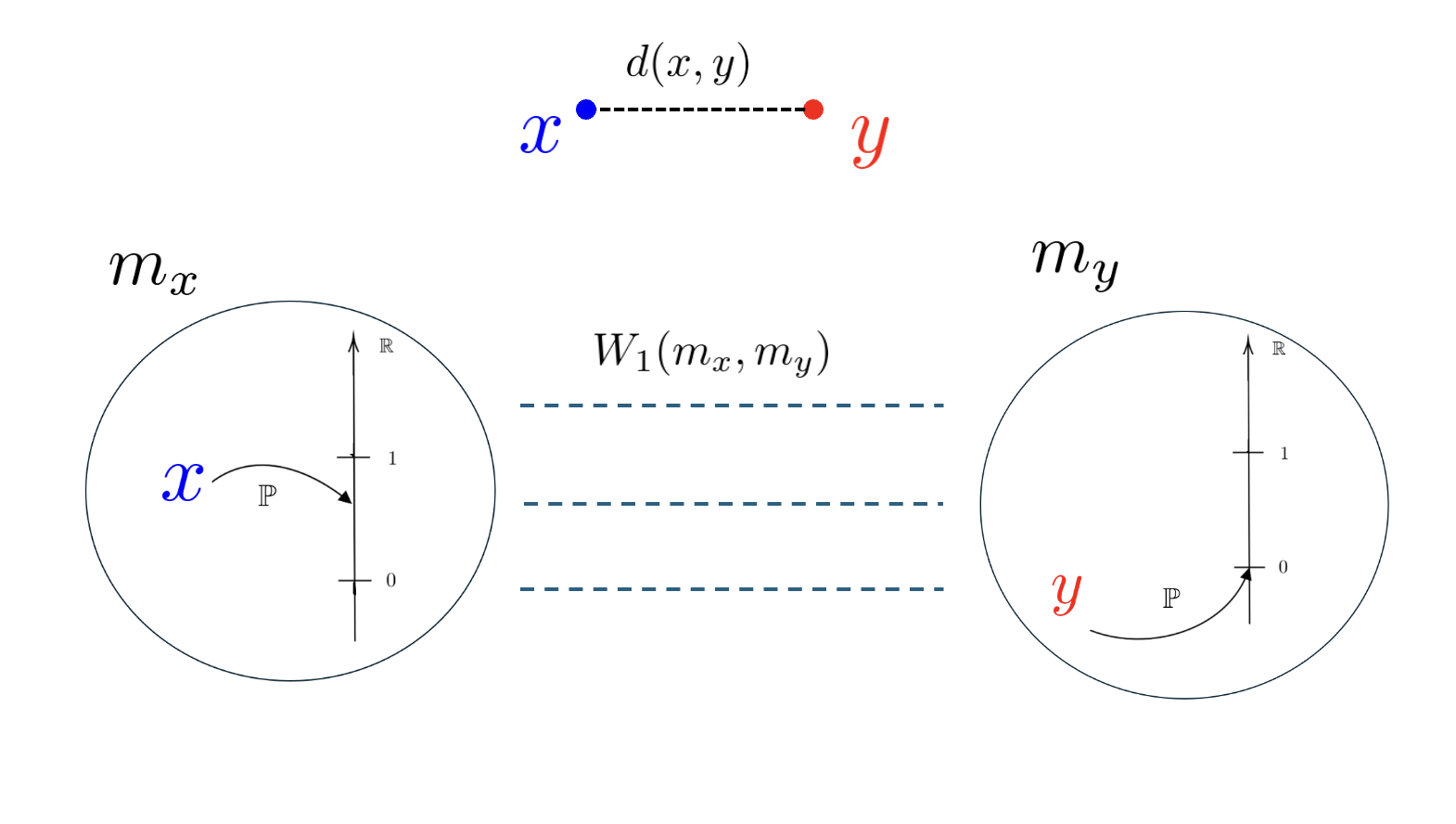}
    \caption{Measures $m_x$ and $m_y$ function as analogues to $S_x$ and $S_y$ on a Riemannian manifold. The sign of Ollivier-Ricci curvature (ORC) depends solely on the difference between $d(x,y)$ and $W_1(m_x,m_y)$. In particular, if if $W_1(m_x,m_y) < d(x,y)$ then ORC is positive, $W_1(m_x,m_y) > d(x,y)$ ORC is negative, and if  $W_1(m_x,m_y) = d(x,y)$  ORC is zero.}
    \label{fig:measures}
\end{figure}
\begin{example}
Consider the case where $(X,d)$ is a Riemannian manifold. The $\epsilon-$random walk on the manifold is $m_x$. For some small $\epsilon$ and  for two points $x$ and $y$ close, then  the $\epsilon$-step random walk is a way of capturing the Ollivier-Ricci curvature in direction $xy$. Geometrically, we can think of $m_x$ as a substitute for the notion of a ball around $x$. 
\end{example}
\begin{definition} (Jump and diffusion constant) 
\begin{itemize}
\item The \textit{jump} of a random walk at $x$ is defined as: 
\[
J(x) := \mathbb{E}_{m_x} d(x, \cdot) = W_1(\delta_x, m_x)
\]
where $\delta_x$ denotes the \textit{Dirac measure} at a point $x$ which assigns mass $1$ to the point $x$ and $0$ at all other points.



\end{itemize}
\end{definition}
These concepts will be essential in a variety of  proofs in this chapter, particularly the analog Bonnet-Myers $L^1$ and $L^2$ theorems, which we present in Theorem \ref{thm: L1-bonnetmyers} and Theorem  \ref{thm: L2-bonnetmyers}. 

\begin{definition} \cite{ollivier_ricci_2009} Define $(X,d)$ as $\epsilon-$geodesic if for any two points $x,y \in X$, there exists a integer $n$ and sequence $x_0 = x, x_1,...,x_n =y$ such that $d(x,y) = \sum d(x_i, x_{i+1})$ and $d(x_i, x_{i+1}) \leq \epsilon$. 
\end{definition}
\begin{remark}
This definition is used in the following proposition. Recall that a geodesic is simply the shortest or minimal length path between two points. Ollivier's paper introduces this definition to capture the idea here is that some spaces can be characterized as near-geodesic, where we can then used local curvature properties to derive global properties such as curvature bounds. In the context of metric spaces, a near-geodesic space is one in which a strict geodesic defining a minimal length path may not always exist. However, such spaces may have paths that are almost minimal paths that are nearly geodesics. This definition allows the flexibility to work with both near and true geodesic spaces. An example of a near-geodesic space, for example, could be a space in which instead of singular shortest paths there are a sequence of paths which could converge to a path of minimal length. Alternatively, there may be spaces in which it is simpler to have approximate shortest paths which may function practically as geodesics do for computational purposes but perhaps the $\epsilon-$geodesics described are not perfectly the shortest paths.
\end{remark}
\begin{remark} With regards to the $\epsilon$ value, Ollivier uses $\epsilon =1$ for  1-dimensional graphs and any given $\epsilon$ for any manifold of a given dimension. 
\end{remark}
The above definition leads us to introduce the following result: 
\begin{theorem} \cite{ollivier_ricci_2009} \label{thm: geodesic}  Suppose that $(X,d)$ is $\epsilon-$geodesic. Then, if 
$$
\operatorname{Ric_O}(x,y) \geq k
$$
for any pair of points $x,y \in X$ with $d(x,y) \leq \epsilon$, then $\operatorname{Ric_O}(x,y) \geq k$ for any pair of points $x,y \in X$. 
\end{theorem}
\begin{proof} Suppose $(X,d)$ is $\epsilon-$geodesic as defined above with sequence $x_0 = x, x_1,..., x_n = y$ similarly defined for integer $n$.  Recall that $W_1$ is a metric on the space of probability measures with finite first moment satisfying the triangle inequality:
$$W_1(\mu, \rho) \leq W_1(\mu, \nu)+W_1(\nu, \rho) \quad
$$ 
for any probability measures $\mu, \nu, \rho
$. 
We can apply this inequality to the measures
$$
m_{x_0}, m_{x_1}, \ldots, m_{x_n}
$$
where $x_0=x, x_n=y$, and $x_1, \ldots, x_{n-1}$ are the intermediate points chosen so that
$$
d(x, y)=\sum_{i=0}^{n-1} d\left(x_i, x_{i+1}\right) \quad \text { and } \quad d\left(x_i, x_{i+1}\right) \leq \varepsilon
$$
By definition of $W_1$ being a metric, we obtain:
$$
W_1\left(m_{x_0}, m_{x_n}\right)=W_1\left(m_x, m_y\right) \leq \sum_{i=0}^{n-1} W_1\left(m_{x_i}, m_{x_{i+1}}\right)
$$
Given that $\operatorname{Ric}_{O} \left(x_i, x_{i+1}\right) \geq k $ for all pairs $\left\{x_i, x_{i+1}\right\}$ with $d\left(x_i, x_{i+1}\right) \leq \varepsilon$ we have:
$$
W_1\left(m_{x_i}, m_{x_{i+1}}\right) \leq(1-k) d\left(x_i, x_{i+1}\right)
$$

and therefore, 
$$
W_1\left(m_x, m_y\right) \leq(1-k) \sum_{i=0}^{n-1} d\left(x_i, x_{i+1}\right)=(1-k) d(x, y)
$$

Therefore $\operatorname{Ric}_{O}(x, y) \geq k$. 
\end{proof}

\begin{example} Suppose an undirected graph \(G=(V,E)\) is $1-$geodesic in the sense that for any vertices $v_1, v_2 \in V$ there exists an integer $n$ and sequence $v_0 = v, v_1,...,v_n =u$ such that $d(v, u) = \sum d(v_i, v_{i+1})$ and $d(u_i, u_{i+1}) \leq 1$, then if $\operatorname{Ric}_{O}(v_1, v_2) \geq k$, then for any $v_1, v_2 \in V$, we have $\operatorname{Ric}_{O}(v_1, v_2) \geq k$ for any pair of vertices in $V$.
\end{example}

We now introduce the following two theorems from Ollivier \cite{ollivier_ricci_2009}. Before we introduce these two results regarding $W_1$ contraction and convergence, we have the following definitions: 
\begin{definition}
If we let $\mathcal{P}(X)$ by the space of all probability measures $\mu$ on $X$ with \textit{finite first moment},  we are saying that for any $o \in X, \int d(o, x) \mathrm{d} \mu(x)<\infty$. Note that for $\mathcal{P}(X)$, the transportation distance $W_1$ is finite and therefore a true distance. 
\end{definition}
\begin{definition}
Let $\mathcal{P}(X)$ by the space of all probability measures $\mu$ on $X$ with finite first moment as defined above. Let $\mu$ be a probability measure on $X$ and define the measure
$$
\mu * m:=\int_{x \in X} \mathrm{~d} \mu(x) m_x
$$
which is the image of $\mu$ by the random walk.
\end{definition}
\begin{theorem} \label{thm: contraction}($W_1$ contraction)  \\
Let $(X, d, m)$ be a metric space with a random walk. Let $\kappa \in \mathbb{R}$. Then we have $\operatorname{Ric}_O(x, y) \geq k$ for all $x, y \in X$, if and only if for any two probability distributions $\mu, \mu^{\prime} \in \mathcal{P}(X)$ we obtain
$$
W_1\left(\mu * m, \mu^{\prime} * m\right) \leq (1- k) W_1\left(\mu, \mu^{\prime}\right)
$$
Futhermore in this case, if $\mu \in \mathcal{P}(X)$ then $\mu * m \in \mathcal{P}(X)$.
\end{theorem}

\begin{theorem} \label{thm: convergence}($W_1$ convergence)\\
Suppose that $\operatorname{Ric}_O(x, y) \geq k>0$ for any two distinct $x, y \in X$. Then the random walk has a unique invariant distribution $v \in \mathcal{P}(X)$.

Moreover, for any probability measure $\mu \in \mathcal{P}(X)$, the sequence $\mu * m^{* n}$ tends exponentially fast to $v$ in $W_1$ distance. Namely
$$
W_1\left(\mu * m^{* n}, v\right) \leq (1- k ^n W_1(\mu, v)
$$
and in particular
$$
W_1\left(m_x^{* n}, v\right) \leq (1- k )^n J(x) / \kappa
$$
\end{theorem}
We cite the proofs in \cite{ollivier_ricci_2009} for reference, but note that Theorem \ref{thm: convergence} follows from the contraction property of \ref{thm: contraction}. Other important results from Ollivier which we will not state here include proofs that the operations of composition, superposition and $L-1$ tensorization all preserve positive Ollivier-Ricci curvature. 

We will now focus on the major proofs of Ollivier that extend the Bonnet-Myers and Levy-Gromov theorems from classical Ricci curvature to Ollivier-Ricci curvature. For the Bonnet-Myers theorems, Ollivier extends Bonnet Myers to present the four results: (1) Weak $L^1$ Bonnet Myers, (2) Average $L^1$ Bonnet Myers, (3) Strong $L^2$ Bonnet Myers, and (4) Average $L^2$ Bonnet Myers for Ollivier-Ricci curvature. We will present and prove (1) and (4). Regarding the use of the terms $L^1$ and $L^2$, this is referring to the difference between measuring expected value in the $L^1$ case versus variance or energy in the $L^2$ case, i.e.:
\[
\|f\|_{L^1} = \int |f(x)| \, dx
\]
\[
\]

\[
\|f\|_{L^2} = \left( \int |f(x)|^2 \, dx \right)^{1/2}
\]
Below is what Ollivier describes as a weak analogue to the classical Bonnet-Myers theorem.
\begin{theorem} \label{thm: L1-bonnetmyers}($L^1$ Bonnet-Myers) \\ Suppose that $\operatorname{Ric_O}(x, y) \geq k >0$ for all $x, y \in X$. Then for any $x, y \in X$ we have
$$
d(x, y) \leq  \frac{J(x)+J(y)}{\operatorname{Ric_O}(x, y)}
$$
and therefore
$$\operatorname{diam} X \leq \frac{2 \sup _x J(x)}{k}
$$
\end{theorem}
\begin{proof}
Recall that for a ponint $x \in X$, the expected distance or jump length from x when the random walk begins is defined as follows:
$$
J(x) = \int d(x,y) d \delta_x = W_1(\delta_x, m_x)
$$
Observe that if we compute the 1-Wasserstein distance between two Dirac measures this is equivalent to the distance $d(x,y)$ between the two points these Dirac measures are supported: 
$$
W(\delta_x, \delta_y) = \int d(u,v) d \delta_{(x,y)}(u,v) = d(u,v)
$$
which is intuitive because the cost of transporting a Dirac measure from one point to the other only requires transport from their respective points themselves. From here, we can apply the following triangle inequality as done previously: 
$$
W_1(\delta_x, \delta_y) \leq W_1(\delta_x, m_x) + W_1(m_x, m_y) + W_1(m_y, \delta_y)
$$
which implies that, 
$$
W_1(\delta_x, \delta_y) \leq J(x) + W_1(m_x, m_y) + J(y)
$$
Now recall that we are given the fact that $\operatorname{Ric_O}(x, y) \geq k >0$ for all $x, y \in X$ such that 
$$
1 - \frac{W_1 (m_x, m_y)}{d(x,y)} \geq k >0
$$
So therefore, 
$$
W_1(m_x, m_y) \leq (1-k)d(x,y)
$$
We then obtain the inequality:
\begin{align*}
d(x,y) = W_1(\delta_x, \delta_y &\leq J(x) + W_1(m_x, m_y) + J(y)  \\ &\leq J(x) + J(y) + (1-k)d(x,y)
\end{align*}
and it follows that: 
$$
d(x,y) \leq \frac{\operatorname{Ric}_O}{J(x) + J(x)}
$$
For the second part of this proof we which to show from the previous result that diameter of the space $X$ is bounded by $\frac{2sup_x J(x)}{k}$. By definition: 
$$
\operatorname{diam}X := \operatorname{sup}_{x,y \in X} d(x,y)
$$
Therefore, 
$$
\operatorname{diam} X = \sup_{x, y} d(x, y)
\leq \sup_{x, y} \left( \frac{J(x) + J(y)}{k} \right)
$$
but we arbitrarily taking the supremum for all $x,y$ so the maximal value of $J(x) + J(y)$ will just be $2(J(x))$ for whichever $x$ achieves this upper bound. And we have, 
$$
\operatorname{diam} X \leq 2 \frac{\operatorname{sup}_x J(x)}{k}
$$
as desired. 
\end{proof}
\begin{remark}
Ollivier calls this a weaker analogue to Bonnet-Myers because it is not as tight of a diameter bound as the traditional Bonnet-Myers introduced earlier in the chapter. However, below is an example where this bound is actually sharp where the equality is attained. 
\end{remark}
\begin{example} \cite{ollivier_ricci_2009} We can consider the discrete cube defined as $C= \{0,1\}^N$ where each cube edge has length 1 (equipped with the Hamming metric). Define $m$ to be the lazy random walk on the graph $C$. If we define the lazy random walk such that staying at any vertex occurs with probability $\frac{1}{2}$, then if we let $m$ be the lazy random walk on the graph $C$ we get $m_x(x) = \frac{1}{2}$ and for neighbor $y$ of $x$ we get: $m_x(y) = \frac{1}{2N}$. To compute the Ollivier-Ricci curvature, we will define each vertex of the graph in terms of a string of bits, and then estimate the optimal coupling and obtain a lower bound on the $W_1(m_x, m_y)$ transport cost. 

\begin{proof}
Suppose vertex $u \in C(V)$ can be defined as a bit $000..0$, and $v \neq u$  as $100..0$. For any vertex $u$ of the cube such that $1 \leq k \leq N$  we can denote any neighbor of vertex $u$ as a bit where the $i$th value of the bit string is swapped from $0$ to $1$ or vice versa and is therefore distinct by 1 bit. We must consider how to define an optimal coupling between $m_u$ and $m_v$. We define $m_u$ as follows: the mass at $u$ is $\frac{1}{2}$ and $\frac{1}{2N}$ for all $N$ neighbors of $u$ from $u^i$ to $u^N$ by definition of the lazy random walk. We have the same distribution around $v$, i.e. $m_v$ is defined such that the mass at $v$ is $\frac{1}{2}$ and $\frac{1}{2N}$ for all $N$ neighbors. If the two vertices differ than more bits than the first bit such that $i \geq 2$ for $u^k$, $v^k$, then we can transport mass $\frac{1}{2} - \frac{1}{2N}$ from $x^i$ to $y^i$. The Ollivier Ricci must be $\operatorname{Ric (x,y)}\geq \frac{1}{N}$ given that we are moving only from adjacent neighbors to others with distance 1. We can define the Lipschitz function
$$
f: X \rightarrow \{0,1\}
$$
such that some $w \in X$ gets mapped to the first bit of it's string, and we have $f(x) = 0$, $f(y) = 1$ for example. In order to obtain optimality of this coupling through a lower bound on $W_1(m_x, m_y)$ we can consider the expectations of $f$ under $m_x, m_y$ as follows:
\[
\begin{aligned}
\mathbb{E}_{m_x}[f] &= \frac{1}{2N} \cdot 1 + \left(1 - \frac{1}{2N} \right) \cdot 0 = \frac{1}{2N} \\
\mathbb{E}_{m_y}[f] &= \frac{1}{2} + \frac{1}{2N} \cdot 0 + \left(1 - \frac{1}{2} - \frac{1}{2N} \right) \cdot 1 = 1 - \frac{1}{2N}
\end{aligned}
\]

\[
\begin{aligned}
W_1(m_x, m_y) &\geq \mathbb{E}_{m_y}[f] - \mathbb{E}_{m_x}[f] \\
&= \left(1 - \frac{1}{2N} \right) - \frac{1}{2N} = 1 - \frac{1}{N}
\end{aligned}
\]
Recall that the diameter of a graph is the, in informal terms, the longest-shortest distance between any two vertices of a graph $x,y$. In other words, out of all of the shortest paths in a graph between any two vertices, the diameter is the longest of these. Given the results of Theorem ~\ref{thm: L1-bonnetmyers}, we obtain:
$$
\operatorname{diam}C \leq N
$$
Observe by this definition that the diameter of a discrete cube as we have defined has diameter exactly $N$, because if we measure the distance by total possible numbering of differing bits for distinct vertices, we get $N$ differing bits, and have concluded this example. 
\end{proof}
\end{example}


The goal of this remainder of this thesis is to discuss results of Ollivier-Ricci curvature on graphs, so for this reason we refer the interested reader to Ollivier's original paper  for further extensions of classical Ricci curvature-related theorems, such as the extension of the original Levy-Gromov inequality. In Ollivier's Theorem 33 \cite{ollivier_ricci_2009}, he shows an analogous result that for random walks with positive Ollivier-Ricci curvature and controlled diffusion, then Lipschitz functions on the space have controlled Gaussian-like concentration about their mean, with additional proof of exponential decay at the tails. In addition to that fascinating result, Ollivier also proves a modified logarithmic Sobolev inequality to prove this Gaussian concentration, similar to the Sobolev inequality from Bakry-Emery theory. 

\chapter[Lin-Lu-Yau Extension to Undirected Graphs and Combinatorial Bounds]{Lin-Lu-Yau Extension to Undirected Graphs and Combinatorial Bounds}

In Ollivier's paper, the notion of Ricci curvature is readily applicable to a variety of metric spaces including graphs. However, in that paper Ollivier did not formalize his findings in the graph setting to the extent done by Yong Lin, Linyuan Lu, and Shing-Tung Yau in their 2011 paper. Lin-Lu-Yau \cite{lin2011ricci} present novel results of Ollivier-Ricci curvature on graphs, particularly with regard to theoretical bounds. Second, they extend the notion of Ollivier-Curvature to their own novel definition by applying limits to the Ollivier-Ricci notion under a specific condition, and apply this novel notion to general graphs and graph products. Furthermore, they study Ollivier-Ricci curvature for all general $\alpha-$lazy random walks that need not have stationary probability of $\frac{1}{2}$ as in Ollivier's original paper. Yong Lin and Shing-Tung Yau \cite{lin2010ricci} first published a single result regarding Ollivier-Ricci curvature in their 2010 paper on Ricci curvature and eigenvalue estimates in which they prove a combinatorial lower bound. We will introduce this finding, their modification for ease of application of Ollivier-Ricci curvature in the undirected graph setting. Finally, we will introduce the work of Jurgen Jost and Shiping Liu in their 2014 paper proving a variety of additional combinatorial bounds on Ollivier-Ricci curvature of graphs as a follow-up to the work of Lin-Yau. \cite{jost_olliviers_2014}. The formalization of Ollivier-Ricci curvature on graphs as introduced by Lin-Lu-Yau forms the basis of our proposed Ollivier-Ricci curvature notion on directed graphs in the next chapter. Similarly, the bounds proposed by Jost and Liu on undirected graphs are the essential building blocks upon which we will prove simple directed cases of directed Ollivier-Ricci curvature bounds in the next chapter. 

We will now introduce the extension of Ollivier-Ricci curvature defined by Lin-Lu-Yau (2011) formally in Definition \ref{def: lin lu yau notion}. 
\begin{definition} For any given vertex $x \in V(G)$ we will let $\Gamma(x)$ denote the neighborhood of $x$ not including $x$, such that: 
\begin{align*}
    \Gamma (x) &= \{ v; vx \in E(G)\}\\
    \mathcal{N}(x) &= \Gamma (x) \cup \{x\}
\end{align*}
\end{definition}
The definition below for $W_1(m_u, m_v)$ is not novel to this paper, but we we define in precisely for clarity in the setting of undirected graphs: 
\begin{definition} Let a coupling $A$ between measure $m_u$ and measure $m_v$ be defined $A: V \times V \rightarrow [0,1]$ with finite support where $\sum_{x \in V} A(x,y) = m_v$, $\sum_{y \in V} A(x,y) = m_u$. For $d(x,y)$ defined as the edge-count or graph distance between vertices $x,y \in V$, then we define the 1-Wasserstein distance as: 
$$
W_1(m_u, m_v) = \inf_{A} \sum_{x,y \in V} A(x,y) d(x,y)
$$
in which the sum is taken over all couplings between $m_u, m_v$, equivalent to
$$
W_1(m_u, m_v) = \sup_{f} \sum_{x \in V} f(x) [m_u(x) - m_v(x)]
$$
for all 1-Lipshitx functions $f$.
\end{definition}

\begin{definition} For any vertex $x \in V(G)$ for some undirected graph $G$ and for any $\alpha \in [0,1]$, define the probability measure $m_x^\alpha$ as follows: 
\[
m_x^{\alpha}(v) =
\begin{cases}
\alpha, & \text{if } v = x, \\
\frac{(1 - \alpha)}{d_x}, & \text{if } v \in \Gamma(x), \\
0, & \text{otherwise}.
\end{cases}
\]
Then, for any $x,y \in V(G)$, Lin-Lu-Yau define the $\alpha-$(Ollivier)-Ricci curvature $\operatorname{Ric}_\alpha$ as: 
$$
\operatorname{Ric}_\alpha = 1 - \frac{W(m_x^\alpha, m_y^\alpha)}{d(x,y)}
$$
where the 1-Wasserstein distance is defined as done in Chapters 1 and 2.  
\end{definition}
\begin{remark}
The authors call this notion $\kappa_\alpha$ but we will use $\operatorname{Ric}_\alpha$ for continuity and clarity.
\end{remark}
\begin{remark}
Observe that other than specifically focusing Ollivier's notion towards the case of graphs, Lin-Lu-Yau modify the measure such that depending on the probability input $\alpha \in \mathbb{R}$, the random walk starting at $x$, $m_x^\alpha$, is defined distinctly at $x$ compared to neighbors of $x$, and is furthermore assigned to be zero outside of the neighborhood of $x$. This follows the idea of an  lazy random walk in which the probability of being stationary at $x$ is greater than zero. In this case we clearly observe that the measure defined at $x$ for a random walk starting at $x$ will be exactly $\alpha$. 
\end{remark}
We will introduce the following results from this paper: a theorem on concavity, and selected upper and lower bounds for the curvature of a given edge depending on transport costs. 

\begin{theorem} \label{thm: concavity}(Concavity) $\operatorname{Ric}_\alpha(x,y)$ is concave in $\alpha \in [0,1]$ for any two vertices $x,y \in V(G)$. \\
We define concavity as follows defined:
$$
\operatorname{Ric}_\beta \geq \lambda \operatorname{Ric}_\alpha + (1- \lambda) \operatorname{Ric}_\gamma
$$
for $0 \geq \alpha < \beta < \gamma \leq 1$ where
$\lambda = \frac{(\gamma - \beta)}{(\gamma - \alpha)}$ such that $\beta = \lambda\alpha + (1- \lambda)\gamma$.
\end{theorem}
\begin{remark} The idea here if you have two different random walk processes on the graph, for example one with parameter \(\alpha\) and another with a larger parameter \(\gamma\) and you create an intermediate process by mixing these two using some convex combination with parameter \(\beta\), the resulting curvature \(\operatorname{Ric}_\beta(x,y)\) is at least as high as the weighted average of the curvatures at the endpoints. This concavity indicates a kind of curvature stability where small changes in the random walk’s behavior as controlled by \(\alpha\) do not cause unexpectedly low curvature values. Instead, the intermediate curvature is always at least what you would expect if you compute some kind of average of extreme cases of parameters. 
\end{remark}
We will now move on to a proof of Theorem ~\ref{thm: concavity}. 
\begin{proof}
Define couplings $P, Q$ as the couplings which satisfy: 
$$
W_1(m_x^\alpha, m_y^\alpha) = \inf_{P} \sum_{u,v \in V} P(x,y) d(x,y)
$$
$$
W_1(m_x^\gamma, m_y^\gamma) = \inf_{Q} \sum_{u,v \in V} Q(x,y) d(x,y)
$$
where the $P$ is the coupling between probability measures $m_x^\alpha$ and $m_y^\alpha$ and $Q$ is the coupling between probability measures $m_x^\gamma$ and $m_y\gamma$. 

Let $R = \lambda P + (1- \lambda)Q$. We will verify that $R$ is an optimal coupling between $m_x^\beta$ and $m_y^\beta$ by evaluating $\sum_{u \in V} R(u,v)$ as follows: 
\begin{align*}
\sum_{u \in V} R(u, v) = \sum_{u \in V} \lambda P(u, v) + (1 - \lambda) Q(u, v) \\
= \lambda m_y^\alpha(v) + (1 - \lambda) m_y^\gamma(v) \\
\end{align*}
We wish to prove that $\lambda m_y^\alpha(v) + (1 - \lambda) m_y^\gamma(v) =  m_y^\beta(v)$ use Lin-Lu-Yau definition of probability measure depending on the three  cases of $v=y$ and $v \in \Gamma(y)$, and $v \neq y, \notin \Gamma(y)$. \\
Case 1: $v = y$:
\begin{align*}
\lambda m_y^\alpha(v) + (1 - \lambda) m_y^\gamma(v) &= \lambda \alpha + (1 - \lambda)\gamma \\
&= \beta \\
&= m_y^\beta(v)
\end{align*}
Case 2: $v \in \Gamma(y)$:
\begin{align*}
\lambda m_y^\alpha(v) + (1 - \lambda) m_y^\gamma(v) = \lambda \frac{1 - \alpha}{d_y} + (1 - \lambda)\frac{1 - \gamma}{d_y} \\
= \frac{1 - \beta}{d_y} \\
= m_y^\beta(v)
\end{align*}
Case 3: $v \neq y, \notin \Gamma(y)$:
\begin{align*}
\lambda m_y^\alpha(v) + (1 - \lambda) m_y^\gamma(v) = \lambda (0) + (1 - \lambda) (0) = 0 = m_y^\beta(v)
\end{align*}
and therefore $\sum_{u \in V} R(u, v) = m_y^\beta(v)$ in all cases. We can repeat each of these cases analogously to show that $\sum_{v \in V} R(u, v) = m_x^\beta(u)$, and therefore $R$ is a coupling between $m_x^\beta(u)$ and $m_y^\beta(v)$, we have by definition that: 
$$
W_1(m_x^\beta(u), m_y^\beta(v)) \leq \sum_{u, v \in V} R(u, v) d(u,v)
$$
Furthermore, by our previous work we can obtain the following equality by the definition of $R$: 
\begin{align*}
\sum_{u, v \in V} R(u, v) d(u,v)  
&= \lambda \sum_{u,v \in V} P(u, v) d(u, v) + (1 - \lambda) \sum_{u,v \in V} Q(u, v) d(u, v) \\
&= \lambda W(m_x^\alpha, m_y^\alpha) + (1 - \lambda) W(m_x^\gamma, m_y^\gamma)
\end{align*}
By substitution, 
$$
W(m_x^\beta, m_y^\beta) \leq \lambda W(m_x^\alpha, m_y^\alpha) + (1 - \lambda) W(m_x^\gamma, m_y^\gamma)
$$
Therefore because we have, 
$$
\kappa_\beta(x, y) = 1 - \frac{W(m_x^\beta, m_y^\beta)}{d(x, y)} 
$$
It follows that, 
\begin{align*}
\kappa_\beta(x, y) &\geq \lambda \left( 1 - \frac{W(m_x^\alpha, m_y^\alpha)}{d(x, y)} \right) + (1 - \lambda) \left( 1 - \frac{W(m_x^\gamma, m_y^\gamma)}{d(x, y)} \right) \\
&= \lambda \kappa_\alpha(x, y) + (1 - \lambda) \kappa_\gamma(x, y)
\end{align*}
and the proof is complete. 
\end{proof}
\begin{remark} 
This lemma indicates that any function
$$
\varphi(\alpha) = \frac{\operatorname{Ric}_\alpha(x,y)}{(1-\alpha)}
$$
is increasing on $\alpha$ over the interval $[0,1]$. 
\end{remark}
In order to make an important observation about the Lin-Lu-Yau version of Ollivier Ricci curvature regarding this result, we will state another important result of the paper. 
\begin{theorem} \label{yau-bound1} For any $x,y \in V$ and $\alpha \in [0,1]$, we obtain the following upper bound on $\operatorname{Ric}_\alpha$:
$$
\operatorname{Ric}_\alpha \leq (1- \alpha) \frac{2}{d(x,y)}
$$
\end{theorem}

\begin{definition} \label{def: lin lu yau notion} (Lin-Lu-Yau Graph Ricci Curvature): By the fact that $\varphi(\alpha)$ as defined above is increasing on $\alpha \in [0,1]$ due to Theorem \ref{thm: concavity} and by the bound given by Theorem \ref{yau-bound1}, then Lin-Lu-Yau show that the following limit exists: 
$$
\lim_{\alpha \to 1} \frac{\operatorname{Ric}_\alpha(x, y)}{1 - \alpha}
$$
and they call the result of the limit the Ricci curvature of edge $xy \in E(G)$. For clarity, we denote this: 
$$
\operatorname{Ric_{LLY}} = \lim_{\alpha \to 1} \frac{\operatorname{Ric}_\alpha(x, y)}{1 - \alpha}
$$
to specify this as the \textit{Lin-Lu-Yau graph Ricci curvature} in an extension of Ollivier's. 
\end{definition}

We will now introduce an result of Lin-Lu-Yau that functions as an analog of Ollivier's $L^1$ Bonnet-Myers result presented in Theorem \ref{thm: L1-bonnetmyers} after introducing a lemma of the paper required for the proof. Recall that the diameter of a graph is the maximal distance between any two distinct vertices of the graph. 

\begin{lemma} \label{lemma2.3}
If \( \operatorname{Ric_{LLY}}(x, y) \geq k \) for any edge \( xy \in E(G) \), then \( \operatorname{Ric_{LLY}}(x, y) \geq k \) for any pair of vertices \( (x, y) \).
\end{lemma}

See Lemma 2.3 of \cite{lin_ricci_2011} for a proof of this statement. 

\begin{theorem} ($L^1$ Bonnet-Myers for Graphs)\\ 
For any \( x, y \in V(G) \), if \( \operatorname{Ric_{LLY}}(x, y) > 0 \), then
\[
d(x, y) \leq \left\lfloor \frac{2}{\operatorname{Ric_{LLY}}(x, y)} \right\rfloor.
\]
Moreover, if for any edge \( xy \), \( \operatorname{Ric_{LLY}}(x, y) \geq k > 0 \), then the diameter of graph \( G \) has the following upper bound:
$$
\operatorname{diam}(G) \leq \frac{2}{k}.
$$
\end{theorem}
\begin{proof}
By Theorem \ref{yau-bound1}, we have
\[
\frac{\kappa_\alpha(x, y)}{1 - \alpha} \leq \frac{2}{d(x, y)}.
\]
To make this of the LLY form, we take the linmit $\alpha \rightarrow 1$ and obtain the first part below, where the $d(x,y)$ follows from the fact that \( \kappa(x, y) > 0 \).
\[
\kappa(x, y) \leq \frac{2}{d(x, y)} \implies d(x, y) \leq \frac{2}{\kappa(x, y)} 
\]
The assumption here is that $\operatorname{Ric_{LLY}}(x, y) > k > 0 $ for all edges we can invoke Lemma \ref{lemma2.3} to make the claim therefore that for any pair of vertices in $x,y \in V(G)$ we obtain: 
$$
d(x,y) \leq \frac{2}{k}
$$
Since $\operatorname{diam}(G) = \operatorname{max}_{x,y \in V(G)} d(x,y)$ it is then clear that we obtain: 
$$
\operatorname{diam}(G) \leq \frac{2}{k}
$$
\end{proof}

\section*{Combinatorial Bounds of Jost-Liu}

We will first introduce the 2010 result of Yong Lin and Shing-Tung Yau on Ollivier's Ricci curvature. This result is the basis of the combinatorial curvature bounds of Jost-Liu which we will present in this remainder of this chapter. 
\begin{theorem} (Lin-Yau Bound for Ollivier-Ricci on Locally Finite Graph) \cite{lin2010ricci} \label{lin-yau} For an infinite but locally finite, weighted, and connected graph G with adjacent vertices $x,y \in V$ (connected by an edge), and furthermore $d_x$ denoting the degree of $x$ and $d_y$ the degree of $y$, Ollivier-Ricci curvature has the following lower bound for $d_x, d_y > 1$:
$$
\operatorname{Ric_O}(x,y) \geq \frac{2}{d_x} + \frac{2}{d_y} - 2
$$
\end{theorem}

We will introduce a proof of a variant of this statement given by Liu-Jost and an adaption of  their proof:
\begin{theorem} \textbf{Jost and Liu} \cite{jost_olliviers_2014}
On a locally finite graph $G = (V,E)$,  for any pair of neighboring vertices $x, y$ and $(A)_+ := \operatorname{max}(A,0)$, then:
\[
\operatorname{Ric_O}(x, y) \geq -2 \left( 1 - \frac{1}{d_x} - \frac{1}{d_y} \right)_+ = 
\begin{cases} 
-2 + \frac{2}{d_x} + \frac{2}{d_y}, & \text{if } d_x > 1 \text{ and } d_y > 1; \\ 
0, & \text{otherwise}.
\end{cases}
\]
\end{theorem}
\begin{proof}
Let $d(x,y) =1$. Then the Ollivier-Ricci curvature computation simplifies by definition to: $$
\operatorname{Ric_O} = 1 - W_1(m_x, 
m_y)
$$
We apply the Kantorvich Rubenstein duality that was introduced in Chapter 1 and take the following supremum over 1-Lipschitz functions $f$: 
\begin{align*}
W_1(m_x, m_y) &= \sup_{f, 1\text{-Lip}} \left( \frac{1}{d_x} \sum_{z, z \sim x} f(z) - \frac{1}{d_y} \sum_{z', z' \sim y} f(z') \right) \\
\end{align*}
We can use the simple trick of re-centering through $f(x), f(z)$ subtraction from $f(z)$ and the add back in to get closer to the form in the bound:
\begin{align*}
W_1(m_x, m_y) &= \sup_{f, 1\text{-Lip}} \left( \frac{1}{d_x} \sum_{z, z \sim x, z \ne y} (f(z) - f(x)) - \frac{1}{d_y} \sum_{z', z' \sim y, z' \ne x} (f(z') - f(y)) \right. \\
&\quad \left. + \frac{1}{d_x} (f(y) - f(x)) - \frac{1}{d_y} (f(x) - f(y)) + (f(x) - f(y)) \right) 
\end{align*}
where by definition of 1-Lipschitz functions we obtain:
\begin{align*}
 &\le \left| \frac{d_x - 1}{d_x} + \frac{d_y - 1}{d_y} \right| + \left| 1 - \frac{1}{d_x} - \frac{1}{d_y} \right| = 2 - \frac{1}{d_x} - \frac{1}{d_y} + \left| 1 - \frac{1}{d_x} - \frac{1}{d_y} \right| \\
&= 1 + 2\left(1 - \frac{1}{d_x} - \frac{1}{d_y} \right)_+.
\end{align*}
which therefore implies
$$
\operatorname{Ric_O}(x, y) \geq -2 \left( 1 - \frac{1}{d_x} - \frac{1}{d_y} \right)_+
$$
\end{proof}
\begin{proposition} All trees achieve this equality. 
\end{proposition}
\begin{proof} Jost-Liu use the argument that given the result of Theorem 2 it is only necessary to prove that:
 \[
W_1(m_x, m_y) \geq 1 + 2\left(1 - \frac{1}{d_x} - \frac{1}{d_y}\right)_{+}
\]
Observe that if either $d_x = 1$ or $d_y = 1$, then $W_1(m_x, m_y) = 1$. Therefore, consider:
$$
1 - \frac{1}{d_x} - \frac{1}{d_y} \geq 0
$$
Our goal is to find a 1-Lipschitz function $f$ on a tree where $f(z) = 0$ if $z$ is a neighbor of $y$ that is not $x$, $f(z) = 1$ for $z = y$, $f(z) = 2$ for $z = x$, $f(z) = 3$  if $z$ is a neighbor of $x$ that is not $y$. Further, observe that any path joining two vertices on a tree is always unique, such that there is no alternative additional path between neighbors of $x$ and $y$ that exists. Thus the functions we listed can be extended to a 1-Lipschitz function for all of $G$. Again by the Kantorovich duality:
\[
W_1(m_x, m_y) \geq \frac{1}{d_x} \left(3(d_x - 1) + 1\right) - \frac{1}{d_y} (2)
\]
\[
= 3 - \frac{2}{d_x} - \frac{2}{d_y}
\]
\end{proof}
 
The last result of Jost-Liu which we will present is the following: 
\begin{theorem} \cite{jost_olliviers_2014}
\textit{Let $G = (V, E)$ be a locally finite graph. For any pair of neighboring vertices $x, y \in V$,}
\[
\kappa(x, y) \geq 
- \left( 1 - \frac{1}{d_x} - \frac{1}{d_y} - \frac{\#(x, y)}{\min\{d_x, d_y\}} \right)_{+}
- \left( 1 - \frac{1}{d_x} - \frac{1}{d_y} - \frac{\#(x, y)}{\max\{d_x, d_y\}} \right)
+ \frac{\#(x, y)}{\max\{d_x, d_y\}}.
\]
\end{theorem}
For their proof, they suppose without loss of generality  that
\begin{equation*}
    d_x = \max\{d_x, d_y\}, \quad d_y = \min\{d_x, d_y\}
\end{equation*}

\begin{figure}
    \centering
    \includegraphics[width=0.7\linewidth]{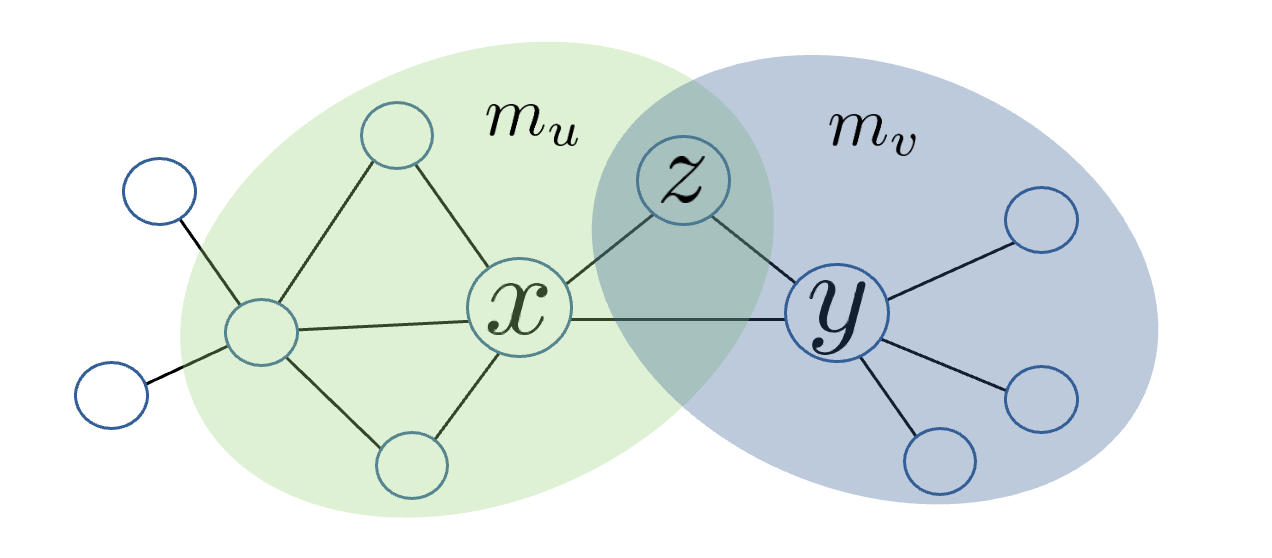}
    \caption{An example type of graph diagram used for transport plan proofs by Jost--Liu in which for an edge $xy$, vertices $x$ and $y$ also share a common neighbor $z$ to form a triangle, and where $x$ and $y$ both have non-shared neighbors. The emphasis of the triangle is used in such proofs to include positive curvature contributions of triangles.}
    \label{fig:jost-liu-transport-plan}
\end{figure}

The above assumption is due to symmetry and the fact that for the transport plan, $d_x$ will be limited by maximum degree of $y$ and $d_y$ limited by the minimum degree of $x$ due to the transport plan of flow from $x$ neighbors to $y$ neighbors. Their idealized transfer plan moving $m_x$ to $m_y$ follows the following procedure: (1) transport the mass of $\frac{1}{d_x}$ from $y$ to $y$'s own neighbors, (2) transport a mass of $\frac{1}{d_y}$ from $x$'s own neighbors to $x$, (3) fill gaps (i.e. redistribute mass) using the mass at $x$'s own neighbors. 

We can then assign the following costs based on edge distances: filling the gaps at common neighbors costs $2$ and the one at $y$'s own neighbors costs $3$. This can be easily visualized by the edge-count graph distance from $x$ neighbors to $y$ and subsequently $x$ neighbors to $y$ neighbors through the shortest path distance going through the edge $xy$. Note that if for any steps in this transport plan there are no common neighbors or $y's$ own neighbors for example, these steps are skipped. Furthermore, observe how this structure of transport plan calculates a mass transport flow from $x$ through to $y$. In the undirected case, this is suitable because by symmetry, the selection of $x$ versus is $y$ is arbitrary following a redistribution of mass, and due to the symmetry of $W_1(m_x, m_y)$. As a proof sketch, form inequalities based on which of the mass transport steps can be performed, and include the count of triangles as follows. For example: they show step (1) is only possible if and only if:
$$
1 - \frac{1}{d_x} - \frac{1}{d_y} - \frac{|T|}{\min\{d_x, d_y\}}
$$
if $|T|$ is the number of triangles for which $x,y$ are both vertices if $T$ is the set of those triangles. They then prove by the following set of inequalities: 
\[
\begin{aligned}
W_1(m_x, m_y) &\leq \frac{1}{d_x} \times 1 + \frac{1}{d_y} \times 1 + \left( \frac{1}{d_y} - \frac{1}{d_x} \right) \times |T| \times 2 \\
&\quad + \left( 1 - \frac{1}{d_x} - \frac{1}{d_y} - \left( \frac{1}{d_y} - \frac{1}{d_x} \right) \times |T| - \frac{1}{d_x} |T| \right)\times 3 \\
&= 3 - \frac{2}{d_x} - \frac{2}{d_y} - \frac{|T|}{d_y} - \frac{2|T|}{d_x}.
\end{aligned}
\]
We refer the curious reader to page 313 of \cite{jost_olliviers_2014} for the complete proof. We now present example computations of graph Ollivier-Ricci curvature that can be done by hand for small graphs, as well as computations using GraphRicciCurvature python package from Ni et al \cite{ni2019community} that validate these sample calculations. 
\section*{Example Computations}
\begin{figure}
    \centering
    \includegraphics[width=1\linewidth]{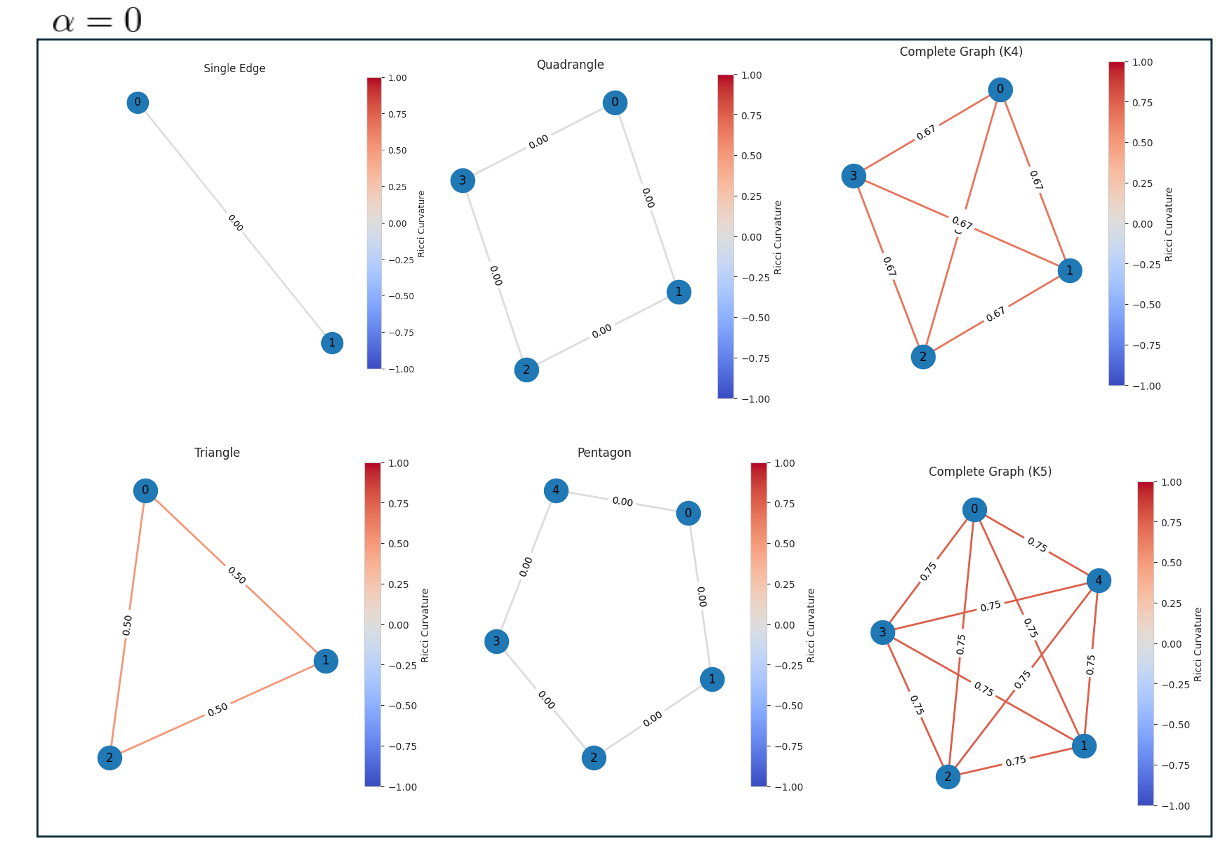}
    \caption{Example graphs with Ollivier--Ricci curvature computed and colored on edges. Here, $\alpha=0$ refers to zero probability of a random walk being stationary at the starting vertex.}
    \label{fig:orc-samples-alpha0}
\end{figure}

\begin{figure}
    \centering
    \includegraphics[width=1\linewidth]{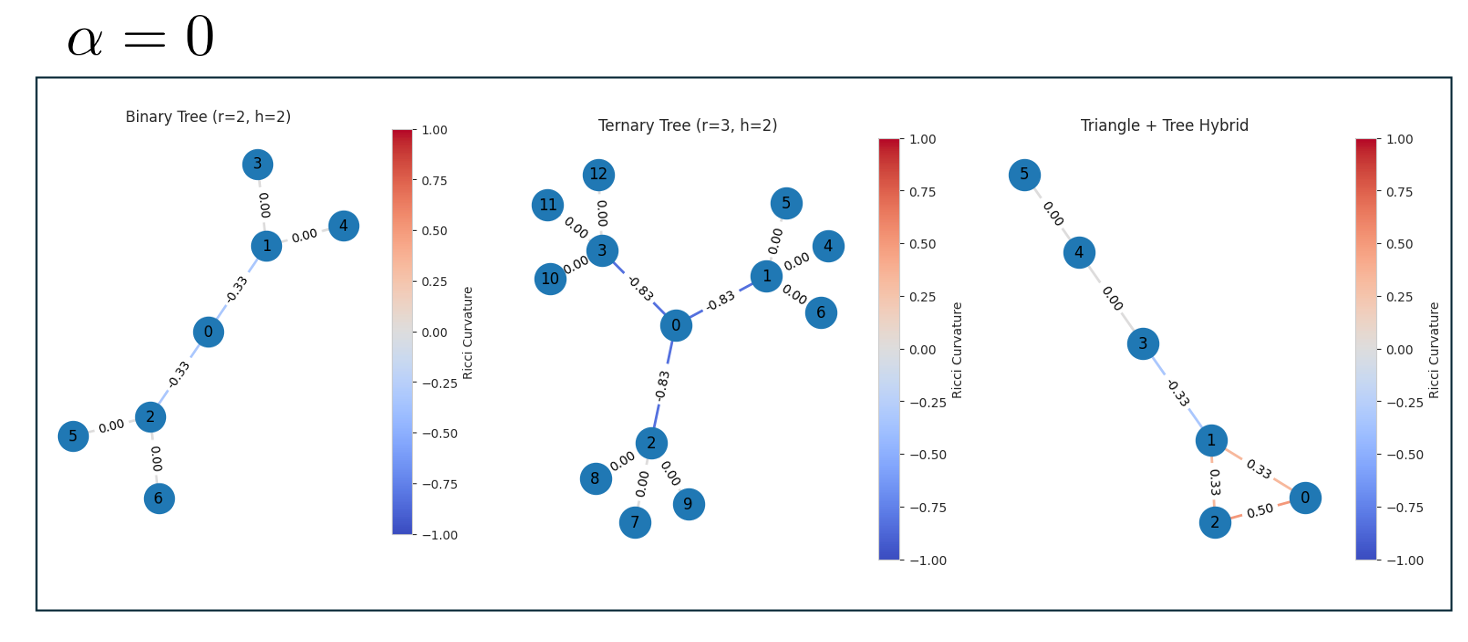}
    \caption{Tree graphs with Ollivier--Ricci curvature computed and colored on edges. Here, $\alpha=0$ refers to zero probability of a random walk being stationary at the starting vertex.}
    \label{fig:orc-tree-samples-alpha0}
\end{figure}

\begin{figure}
    \centering
    \includegraphics[width=1\linewidth]{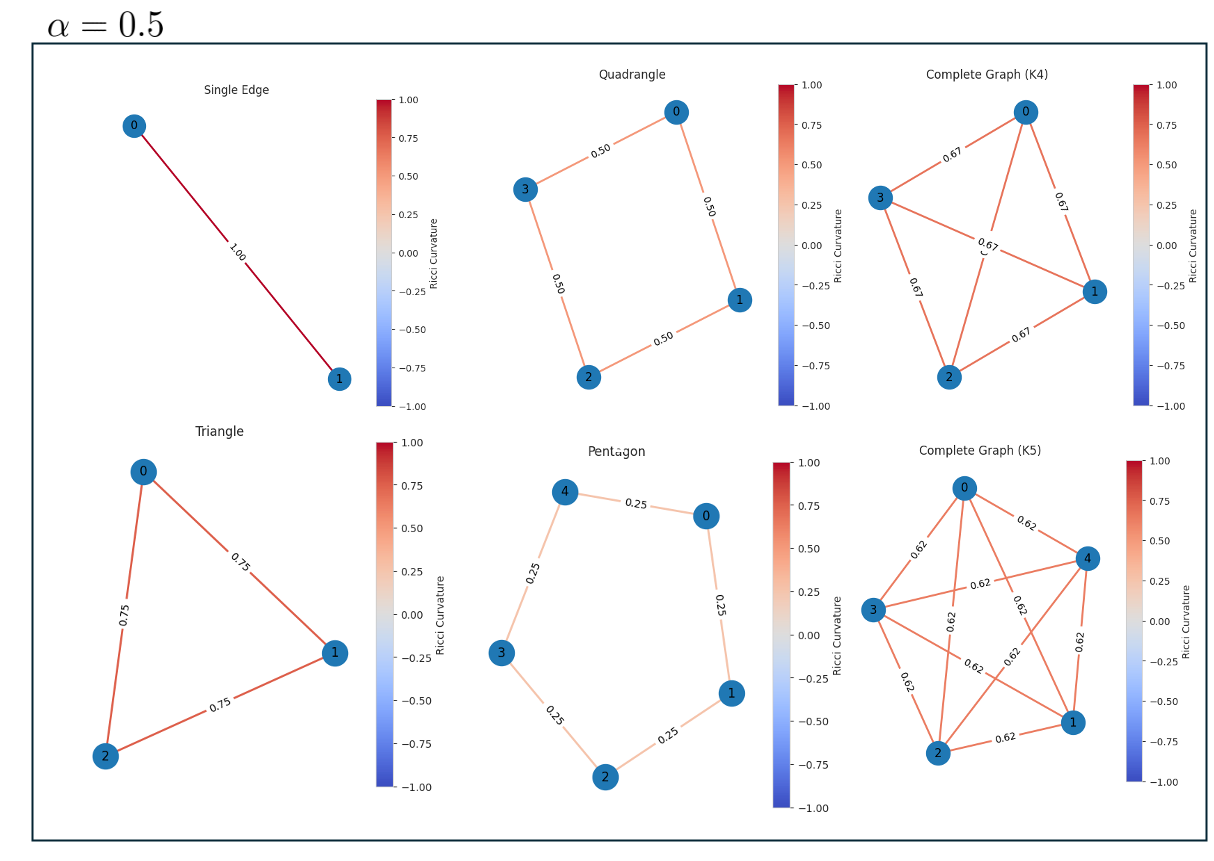}
    \caption{Example graphs with Ollivier--Ricci curvature computed and colored on edges. Here, $\alpha=0.5$ refers to $0.5$ probability of a random walk being stationary at the starting vertex.}
    \label{fig:orc-samples-alpha05}
\end{figure}

\begin{figure}
    \centering
    \includegraphics[width=1\linewidth]{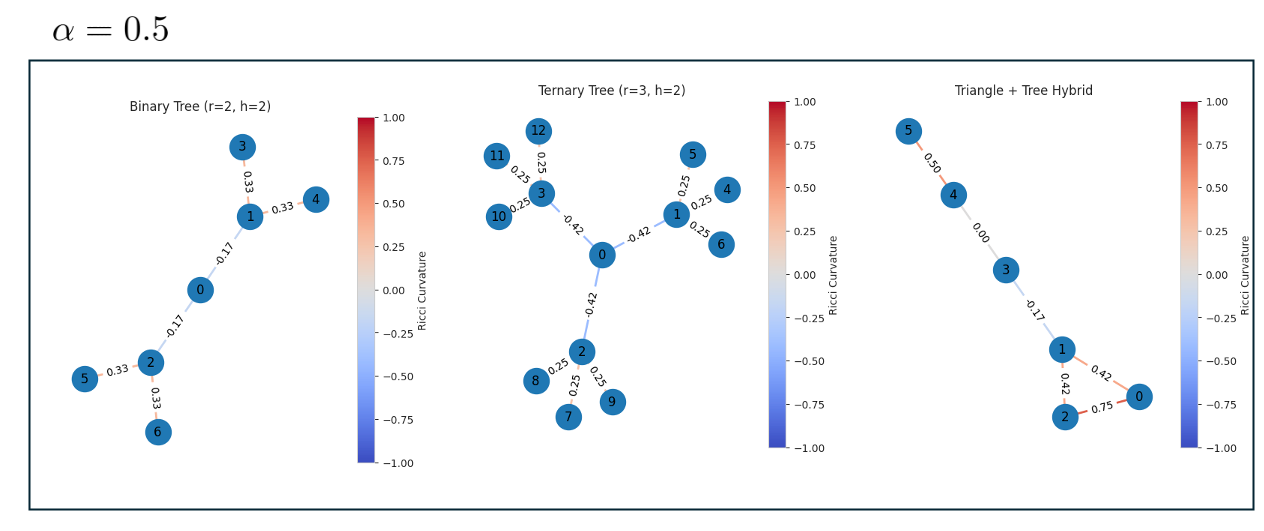}
    \caption{Tree example graphs with Ollivier--Ricci curvature computed and colored on edges. Here, $\alpha=0.5$ refers to $0.5$ probability of a random walk being stationary at the starting vertex.}
    \label{fig:orc-tree-samples-alpha05}
\end{figure}

\begin{example} \label{example1}(Single edge: $\alpha = 0$)
Consider a graph consisting of a single edge connecting two nodes, labeled $0$ and $1$. In this case, each node has only one neighbor. For $\alpha = 0$, the probability measure $m_0$ at node $0$ places all mass on its neighbor: $m_0 = \delta_1$, and similarly $m_1 = \delta_0$. The 1-Wasserstein distance between $m_0$ and $m_1$ is $W_1(m_0, m_1) = 1$, since the full unit of mass must be transported across a single edge of length 1. Thus, the Ollivier--Ricci curvature is $\kappa(0,1) = 1 - 1 = 0$.
\end{example}
\begin{example} (Single edge: $\alpha = 0.5$)
For $\alpha = 0.5$, each node retains half of its probability mass and distributes the other half to its neighbor. Therefore, $m_0 = 0.5\delta_0 + 0.5\delta_1$ and $m_1 = 0.5\delta_1 + 0.5\delta_0$. Since these distributions are identical, the Wasserstein distance is zero, and the resulting curvature is $\kappa(0,1) = 1 - 0 = 1$.
\end{example}

\begin{example}(Triangle: $\alpha = 0$)
Now consider the triangle graph, a 3-cycle with nodes $\{0,1,2\}$ and edges connecting each pair. We examine the curvature on edge $(0,1)$. For $\alpha = 0$, each node distributes its full mass evenly among its two neighbors. Hence, $m_0 = 0.5\delta_1 + 0.5\delta_2$ and $m_1 = 0.5\delta_0 + 0.5\delta_2$. An optimal transport plan moves $0.5$ units of mass from node $1$ to node $0$ (cost $1$), and $0.5$ units of mass from node $2$ to itself (cost $0$), resulting in $W_1 = 0.5$ and curvature $\kappa(0,1) = 1 - 0.5 = 0.5$.
\end{example}
\begin{example}  \label{example2}
(Triangle: $\alpha = 0.5$)
For $\alpha = 0.5$, each node retains half of its mass and distributes the other half equally among its neighbors. The resulting measures are $m_0 = 0.5\delta_0 + 0.25\delta_1 + 0.25\delta_2$ and $m_1 = 0.5\delta_1 + 0.25\delta_0 + 0.25\delta_2$. An optimal transport plan sends $0.25$ units of mass from node $0$ to node $1$ (cost $1$), while all remaining mass matches identically. Therefore, $W_1 = 0.25$, yielding curvature $\kappa(0,1) = 1 - 0.25 = 0.75$.
\end{example}

\chapter{Considerations for Ollivier-Ricci Curvature of Directed Graphs}

\section*{The Challenges of Directionality}

As discussed in previous chapters, a central feature of the 1-Wasserstein distance that allows ease of use on undirected graphs is that it is a true distance metric and is therefore symmetric such that for some vertices $x,y$ we have:
$$
W_1(m_x,m_y) = W_1(m_y,m_x)
$$
In other words, the cost of transporting $m_x$ to $m_y$ is equivalent to the cost in the opposite direction from $m_y$ to $m_x$. However, if you consider the classic transport plan set-up and extend it to directed graphs, then observe that transport costs are not going to be symmetric and necessarily by definition of a directed graph prohibiting flow from $y$ to $x$ on an edge $x \rightarrow y$ we have that for all $x,y \in V(D)$ of a directed graph that:
$$
W_1(m_x,m_y) \neq W_1(m_y,m_x)
$$
Furthermore, we must abide by the directionality of an edge in the transport plan such that if an edge begins at $x$ and terminates at $y$, then no transport can ever occur $y \rightarrow x$ by definition. Furthermore, consider the geometric feautures of a graph utilized by Lin-Lu-Yau and Jost-Liu in the undirected graph settin for proving results regarding Wasserstein distance or Ollivier-Ricci bounds: number of triangles in a graph or whether or not a graph or subgraph are trees are some of the most important examples. In these cases, edge directionality adds an immensely complex layer of asymmetry and, furthermore, classic graph features such as cycles and trees have sub-types such as those shown in the below Figure.

Observe that a $3-$cycle in an undirected graph is simply a triangle and we can treat all triangles the same when computing mass transport. However, we may have a triangle in a \textit{directed} graph that is acyclic in the sense that all three edges do not follow the same direction such that we obtain to colliding edge heads or adjacent tails. We will formalize this idea with regards to \textit{effective length} in this chapter. Furthermore, when it comes to tree-like behavior in directed graphs, the underlying undirected graph may be a tree, however the directed graphs edges can either be all \textit{in-branching} or pointing intowards the root vertex, \textit{out-branching} or all pointing outwards away from the root, or a combination of edges pointing towards and away from a root, completely blocking transfer of mass from a leaf to a root for example. We will formalize the importance of these examples when proving case-by-case bounds for Ollivier-Ricci curvature for directed graphs of different cyclical and tree-like properties. First, however we will discuss the limited current reseedgeh on Ollivier-Ricci curvature for directed graphs.

\begin{figure}
    \centering
    \includegraphics[width=0.75\linewidth]{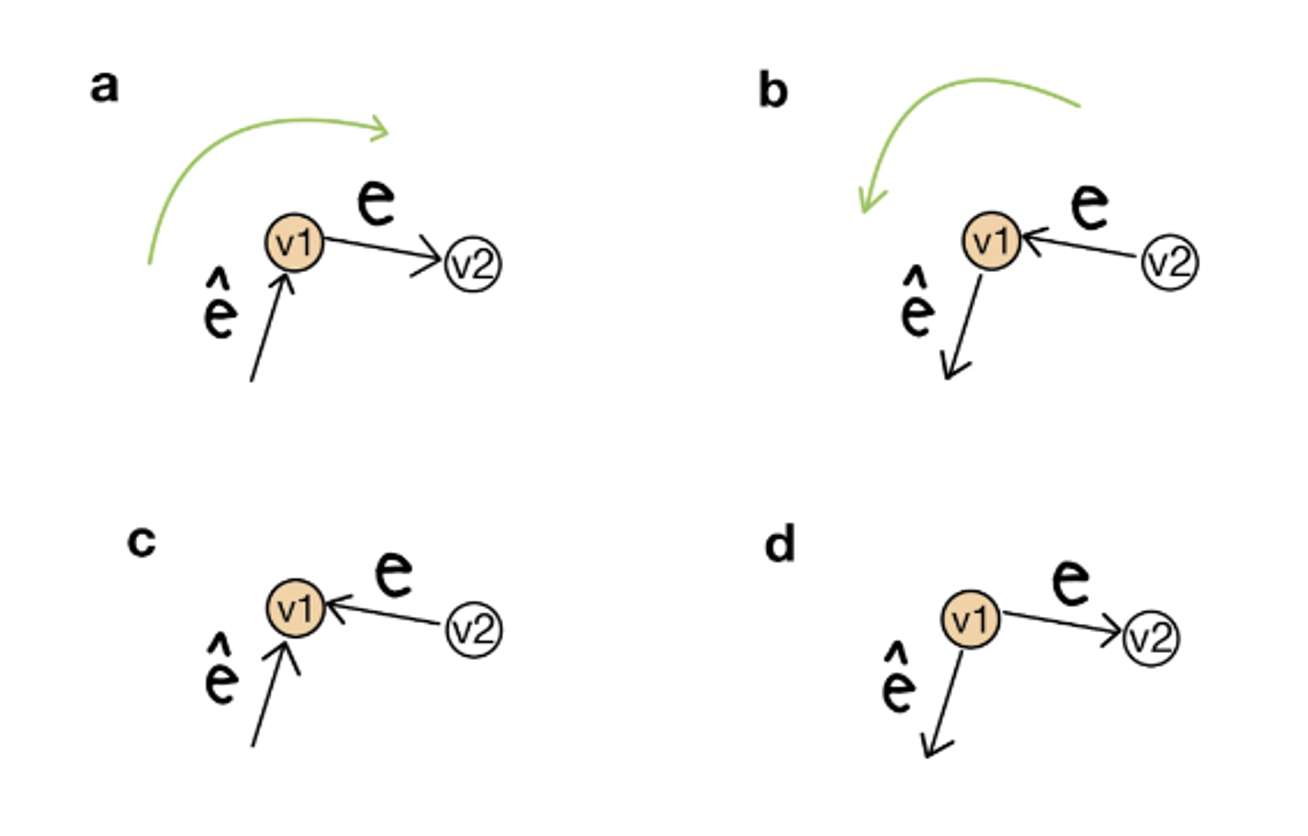}
    \caption{Directionality of edges can increase complexity with regards to mass transport or flow through a vertex, resulting in more possible transport plans and potential for mass blockage, bottlenecks, or sinks. (a) Continuous directed flow from $\hat{e}$ to $e$, (b) Continuous directed flow from $e$ to $\hat{e}$, (c) Restriction to only inward flow to $v_1$, (d) Restriction to only outward flow from $v_1$.}
    \label{fig:directionality-impacts}
\end{figure}
\section*{An Introduction to Directed Graphs}
Directed graphs, known as \textit{digraphs}, are the primary focus of this chapter and connections between digraphs and discrete is the focus of this thesis. Here we formalize necessary notation and definitions regarding digraphs to begin proving relevant results. This chapter and all work in directed graphs seeks to study this crucial difference between undirected edges and directed edges and the implications directionality has for graph structure and graph properties.

\begin{figure}
    \centering
    \includegraphics[width=0.4\linewidth]{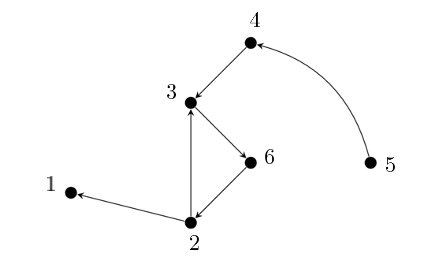}
    \caption{A simple digraph.}
    \label{fig:digraph1}
\end{figure}

\begin{definition}
    A \textbf{directed graph} or \textit{digraph} $D = (V, A)$ is a non-empty finite set $V$ of \textit{vertices} and a finite set $A$ of ordered pairs of distinct vertices called \textit{edges}. We denote the vertex set as $V(D)$ and the edge set of D as $E(D)$ where \textit{edge} $(u,v)$ connecting vertices $u$ and $v$ in a digraph $D$ begins at $u$ and terminates at $v$.
\end{definition}

In the definition above, $u$ acts as the starting vertex so it is called the \textit{tail} of the edge and $v$ is called the \textit{head}, and both $u$ and $v$ are therefore called the \textit{end-vertices} of the edge. There are a few other terms that are used to describe the head and tail of an edge. For example, $u$ is said to dominate $v$. The head and tail of an edge are said to be \textit{adjacent} vertices. It is permissible for a digraph to have an edge $(u,v)$ and an edge $(v,u)$ in which two vertices are connected by opposite edges of opposite direction. Just as with undirected graphs, it is crucial to specify neighborhoods of vertices when working with digraphs. However, if a vertex v is an \textit{end-vertex} of two edges $(u,v)$ and $(v,w)$, then the relationships between its two neighbors $u$ and $w$ are fundamentally different. Here, $u$ is an \textit{in-neighbor} of $v$ because the edge is pointing from the head $u$ and into the tail $v$. We call $w$ an \textit{out-neigbor} of $v$ because the edge is pointing out of $v$ and into $w$. Therefore, for a digraph $D$, the \textit{in-neighborhood} of a vertex $v$ is the set of all \textit{in-neighbors} of $v$ and is denoted as $\mathcal{N}^{in}(v)$ or $\mathcal{N}^{-}(v)$, and  $\mathcal{N}^{in/-}_D(v)$ if we wish to specify the neighborhood is in D. We will use the notation $\mathcal{N}^{in}(v)$ in this thesis. Similarly, the \textit{out-neighborhood} of $v$ is the set of all \textit{out-neighbors} of $v$ and is denoted as $\mathcal{N}^{out}(v)$ or $\mathcal{N}^{+}(v)$. Just as with neighborhoods, we also specify the degree of a vertex based on in-going edges and out-going edges. We define the \textit{in-degree} of a vertex $v$ as the number of edges that point \textit{inwards} into $v$, or in other words, the number of edges for which $v$ is the head. This is denoted by $d^{-}(v)$ or $d^{in}(v)$. We define the \textit{out-degree}, of a vertex $v$ as the number of edges that point \textit{outwards} from $v$, or in other words, the number of edges for which $v$ is the tail. This is denoted by $d^{+}(v)$ or $d^{out}(v)$. 
\begin{definition}
A digraph is \textit{simple} if it does not have \textit{loops} or \textit{multiple edges}. 
\end{definition}
A simple digraph is not permitted to have two or more edges with the same tail and same head, which are called parallel or multiple edges. We would call this type of digraph a \textit{directed multigraph}. Furthermore, it is not permissible for a simple digraph to have a self-loops where a vertex $u$ is both the tail and head of the edge $(u,u)$. We would call this type of digraph a \textit{directed pseudograph}. An \textit{oriented graph} is a special type of digraph that has no pair of edges $(u,v)$, $(v,u)$. Figure X depicts examples of  digraphs, oriented graphs and examples of non-digraphs such as directed multigraphs and directed pseudographs. Much of the theory for directed graphs has been developed by considering manipulations of directed multigraphs and directed pseudographs, so we will discuss at length these types of graphs in this chapter. 

\begin{figure}
    \centering
    \includegraphics[width=0.7\linewidth]{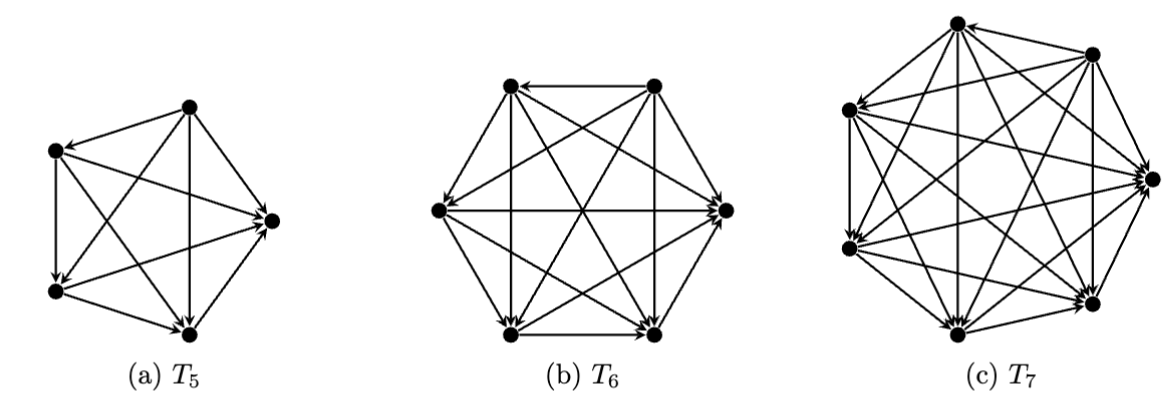}
    \caption{Tournaments}
    \label{fig:tournament}
\end{figure}
Define the graph $D = (V,E)$ to be a directed pseudograph such that self-loops and parallel edges are allowed. Let $a_i \in E(D)$ be edges and $v_j \in V(D)$ be vertices where for a given $a_j$ edge, its tail is defined as $x_j$ is the tail and $x_{j+1} $ is the head for all $j \in [k-1]$. A walk on $D$ is defined as the sequence $W = v_1 a_1 v_2 a_2 v_3....v_{k-1}a_{k-1}v_k$. We call a \textit{walk} or \textit{oriented walk} from $v_1$ to $v_k$ walk a \textit{$(v_1, v_k)-walk$} and call $W$ a \textit{open walk} if $v_1 \neq v_k$, and a \textit{closed walk} if $v_1 = v_k$. The \textit{length} of $W$ is the number of vertices, and the \textit{initial vertex} of $W$ is $v_1$and \textit{terminal vertex} is $v_k$ if $W$ is not closed. If the edges of $W$ are distinct we call this a \textit{trail}, and if the vertices of $W$ are distinct we call this a \textit{directed path} or \textit{dipath}. Now, consider when we may have a directed cycle, or \textit{dicycle}. Let us have the vertices of a dipath $P$ be $V(P) = \{v_1, v_2, v_3,...,v_{k-1}\}$. If all $v \in V$ are distinct and $v_1 = v_k$ we have a cycle only when $k \geq 3$. An \textit{acyclic digraph} contains no dicycles. We call an \textit{ordering} of vertices $v_1,v_2,..,v_n$ an \textit{acylic ordering}. A digraph is defined as a \textit{tournament} $T$ if each pair of distinct vertices $x$, $y$ is connected by exactly one edge, where exactly one of the possible edges $xy$ or $yx$ exists. A tournament which is acyclic is called a \textit{transitive tournament}.

\begin{figure}
    \centering
    \includegraphics[width=1\linewidth]{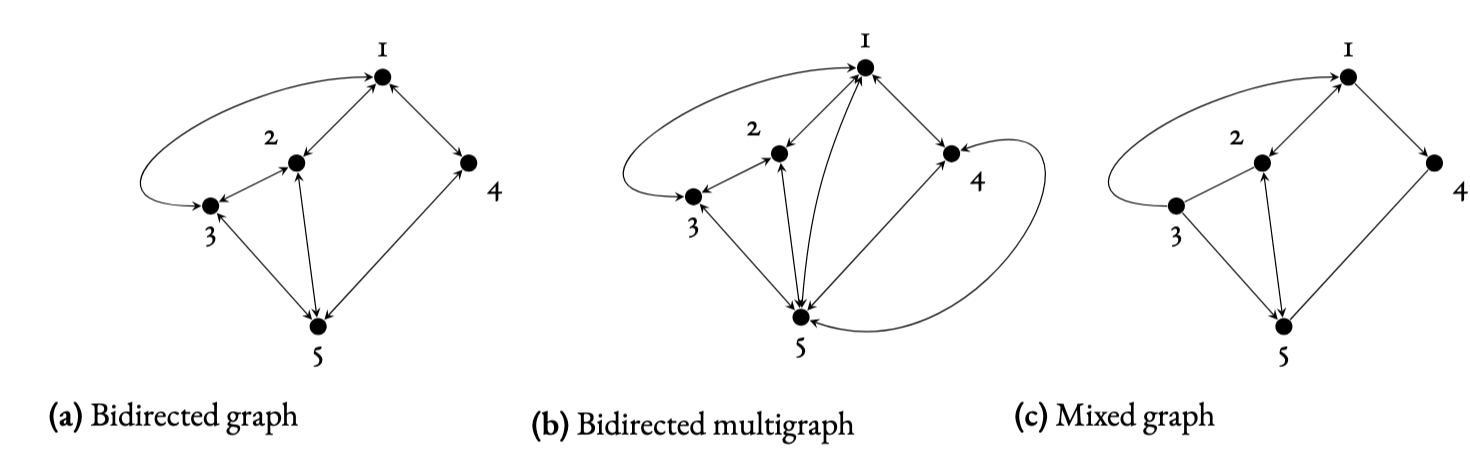}
    \caption{Bidirected graph, bidirected multigraph, and mixed graph. We will discuss primarily only simple directed graphs in this section, but directed networks can be of any number of these forms.}
    \label{fig:bidirected-multigraph-mixed}
\end{figure}

\section*{Existing Notions of Directed Graph ORC}

Because of the challenges in working with directed graphs with the Wasserstein-1 distance, applications of Ollivier-Ricci curvature to directed graphs is very minimally studied despite the extensive reseedgeh published for the undirected graph setting, for example Lin-Lu-Yau and Jost-Liu being some of the earliest to build upon Ollivier's work for undirected graphs as discussed previously. Challenges in asymmetry have made directed graphs similarly understudied for other notions of discrete Ricci curvature, but there are a select few paper such as those on directed Forman-Ricci curvature which explore these ideas \cite{forman_directed}, \cite{paredes2021forman}. To our knowledge, there are only two papers which have attempted to define a notion of Ollivier-Ricci curvature on directed graphs. The first by T. Yamada \cite{yamada2016ricci}, and the second by \cite{olli_forman_directed} as an extension of \cite{forman_directed}. The second of these papers focused on directed networks and argues that directionality can be captured by calculating Ollivier-Ricci curvature as normal for an undirected edge, $e$, but then synthetically assigning a summation of all Ricci curvature values from inward edges to the respective vertex, and doing the same for outward edges \cite{olli_forman_directed}:
\[
\operatorname{Ric_O}_{\text{in}}(v) = \sum_{a \in E_{\text{in},v}} \operatorname{Ric_O}(e);
\quad\quad\quad\quad\quad\quad\quad 
\]
\[
\operatorname{Ric_O}_{\text{out}}(v) = \sum_{a \in E_{\text{out},v}} \operatorname{Ric_O}(e);
\quad\quad\quad\quad\quad\quad\quad 
\]
The paper introducing this notion does not provide any theoretical results on this work. Though it may be a functional notion for some limited algorithmic tasks on networks, we believe it in no way captures true Ollivier--Ricci curvature. Furthermore, it is a vertex-level notion which is not characteristic of Ollivier, and as a result, will fail to represent realistic mass transport through directed edges via a modified 1-Wasserstein distance.

The Yamada paper on directed Ollivier-Ricci curvature \cite{yamada2016ricci} proposes a slight modification of the Lin-Lu-Yau notion of ORC on undirected graphs, and applies this modification to the directed graph setting. This modification is only of the definition of the measure. Recall that Lin-Lu-Yau defines a probability measure $m_x^\alpha$ as follows: 
$$
m_x^{\alpha}(v) = 
\begin{cases}
\alpha & \text{if } v = x, \\
\frac{1 - \alpha}{d_x} & \text{if } v \in \Gamma(x), \\
0 & \text{otherwise}.
\end{cases}
$$
Yamada modifies this to be:
$$
m_x^{\alpha}(v) = 
\begin{cases}
\alpha, & \text{if } v = x, \\
\frac{1 - \alpha}{d_x^{\text{out}}}, & \text{if } (x, v) \in E, \\
0, & \text{otherwise}.
\end{cases}
$$
for application only in well- connected directed graphs. For the probability measure or random walk centered at any given $x$, the modification captures the out-degree $x$, capturing the transport of all mass transport exiting vertex $x$ outwards. We observe that this does not capture in degrees or inward pointing edges entering any given $x$, and therefore fails to capture the a true Ollivier-Ricci curvature notion on directed graphs. In the next subsection of this chapter we will now introduce our novel proposed notion of directed graph Ollivier-Ricci curvature, and various theoretical results that can be found for directed cycles and trees. Furthermore, we will introduce our experimental findings for the novel curvature in Chapter 6 after introducing the respective algorithms and approaches for studying a novel curvature on large directed graphs and networks.

\section*{Proposals and Considerations for Future Work on Directed ORC}

In this sub-section, we present various ideas regarding the potential challenges and feature that a notion for Ollivier-Ricci curvature notion on directed graphs should have. In particular, such a curvature notion should be distinct in many acyclic or tree-like cases from an undirected graph edge, due to the additional constraints and bottlenecks that directed edges impose on optimal transport plans. Note that I present an example notion of a potential measure and example transport plans, but this is in no way comprehensive and requires rigorous further theoretical work to verify the theoretical bounds of the notion given, the connection of the notion to theoretical length, and furthermore, to verify the costs and optimality of various ideas stated in the theoretical transport plans. The most essential aspect of this sub-section is that we provide examples in which asymmetry and the added combinatorial possibilities for directed paths and mass flow on strictly directed graphs (i.e. not mixed or bidirected) that should be considered in future work for the development of a formal method that accurately captures properties of Ollivier-Ricci curvature. 

Note that $\overline{W}$ will no longer be a proper metric to to asymmetry and possibility of not being defined on a directed graph.

For all in-neighbors $z$ of $x$, we can define  the new measure:
$$
m_x(v) = \frac{1}{d_x^{out}}
$$
without the lazy-random-walk and Lin-Lu-Yau notion presented by Yamada. In an analogous by directed method of the mass transport plans presented by Jost-Liu for undirected graphs, we can coordinate a mass transport plan on directed graphs as follows: 

\begin{algorithm}
\caption{Directed Transport Plan}
\begin{algorithmic}[1]
\Procedure{MassTransport}{$x$, $y$, $d_x$, $d_y$}
    \If {$\overrightarrow{xy} \in$ 3-dicyle} 
        \State Move mass of $m_x$ from $y$ to its out-neighbors if they exist.
        \State Move mass of $m_y$ from $x$'s out-neighbors to $x$ if they exist. 
        \State Fill gaps using mass at $x$'s in-neighbors or x,y if not.
        \State \quad Cost for common neighbors is length of \textbf{shortest dipath} from x in-neighbors to common. Cost for $y$'s neighbors is length of \textbf{shortest dipath} from x in-neighbors to y. 
    \Else
        \State Follow identical procedure for acyclic triangles and non-triangles, where no mass transfer occurs against opposite direction edge.
    \EndIf
    No further possible edges to pass mass through. 
\EndProcedure
\end{algorithmic}
\end{algorithm}
where we formalize this by the distribution of measures around $x$. 
\begin{definition}
    For a cycle $C$ in a directed graph $G$, with edges of any direction, the effective length of $C$ is defined as the absolute difference between the number of edges oriented in one direction and the number of edges oriented in the opposite direction. 
\end{definition}
\begin{figure}[htbp]
    \centering
    \begin{tabular}{c}
        \includegraphics[width=.7\textwidth]{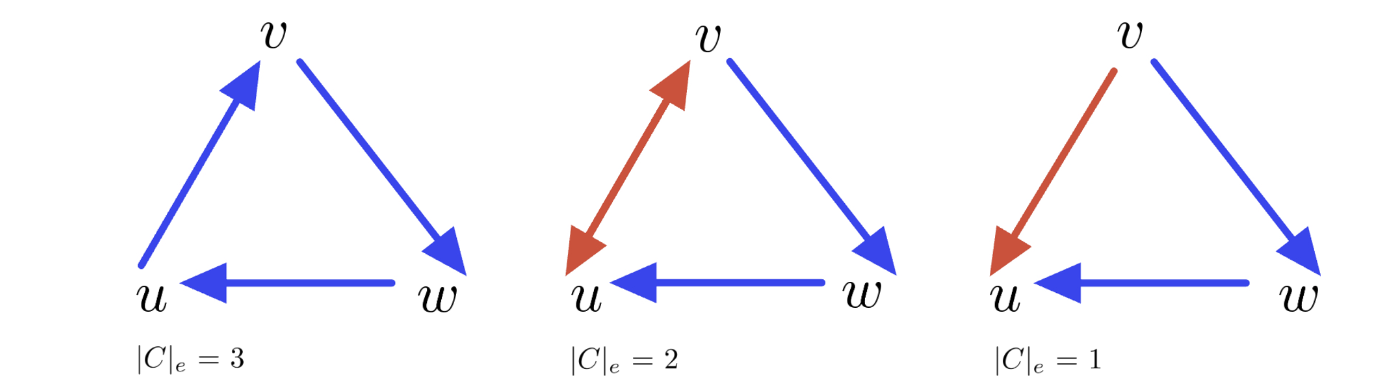} 
    \end{tabular}
    \caption{Effective length 3 (left) through 1 (right) for directed triangles.}
    \label{fig:effective-length-triangles}
\end{figure}
\begin{figure}[htbp]
    \centering
    \begin{tabular}{c}
        \includegraphics[width=.7\textwidth]{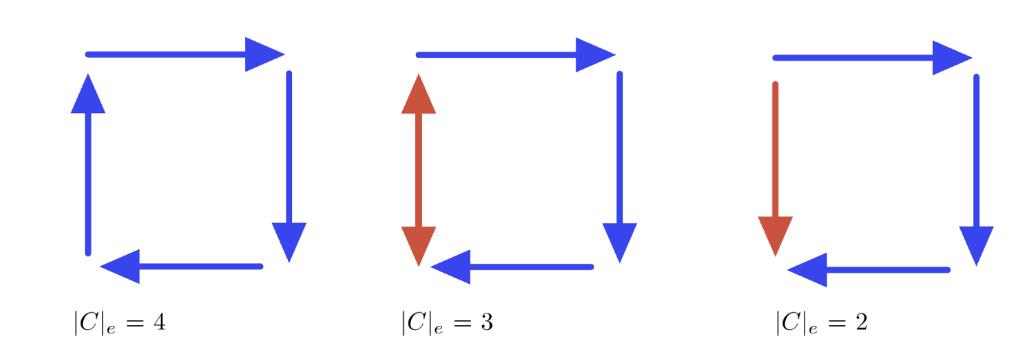} 
    \end{tabular}
    \caption{Variety of effective length values for directed quadrangles.}
    \label{fig:effective-length-quadrangles}
\end{figure}
\noindent

We will adjust this transfer plan. Observe that in the directional case we have extensive combinations of possible transport plan changes even if the number of edges of the origial Jost and Liu example stays fixed.

Observe as depicted in Figure 6 that we will sometimes have to fill mass from $y$ to $y$'s common neighbors with $x$ \textit{directly}. Let's address the first step and consequences of its failure: suppose that $y$ has no out-neighbors. Then mass of $m_x$ cannot be moved from $y$ to neighbors. If $y$ has no neighbors then we are done. If $y$ has only in-neighbors then we can treat this as if $y$ has no neighbors and we are done. Suppose that there is no directed path from $x$ to its common neighbor $z$ and no directed path from $y$ to $z$. Then we cannot transport any mass along this path. Follow the same procedure as $x$. However, if there is a directed path, the length of this new path will be the cost. For example, to transport from an $x$-neighbor in Figure 6 to $z$, the cost is 3, which is the same cost as transporting to a $y$-neighbor from an $x$-neighbor. Therefore, our new transport plan will be entirely based upon scanning possible directed paths. If a dipath exists we compute its length, and if it does not then the cost is insurmountable.

\vspace{0.5cm}

There are two possible cases for $|C_e| = 3$ with only directed edges and not bidirected (i.e., a directed 3-cycle). These are depicted in Figure 5.

\textbf{3-Cycle Cases:}
\begin{enumerate}
    \item $\overrightarrow{xy} \in (xyz)$ directed 3-cycle. 
    \item $\overrightarrow{yx} \in (xyz)$ directed 3-cycle. 
\end{enumerate}
\begin{claim}
So long as $d_{y(out)} > 0$, the cost of moving mass $\frac{1}{d_x}$ from $y$ to $y$'s own  neighbors is $c = 1$. Furthermore, so long as $d_{x(in)} > 0$, the cost of moving mass $\frac{1}{d_y}$ from $x$'s own neighbors to to $x$ is $c = 1$.
\end{claim}
\begin{claim}
Filling gaps at common neighbors with $x$ neighbors is always: 
$$
2 \leq c \leq 3
$$
with cases 1 and 2 as bounds. 
\end{claim}
\begin{claim}
Filling gaps at $y$ neighbors with $x$ neighbors is always: 
$$
3 \leq c \leq 4
$$
with cases 1 and 2 as bounds. 
\end{claim}
Observe that moving mass of $\frac{1}{d_x(in)}$ from y to $y$ neighbors is possible i.f.f.:
$$
1-\frac{1}{d_{x(in)}}-\frac{1}{d_{y(out)}}-\frac{\sharp(C_e=3)}{d_{x}(in) \wedge d_{y (out)}} \geq 0 
$$
The question is, can we suppose WLOG as the Liu Jost do in their proof of Theorem 3 where the suppose $d_y = d_{x} \wedge d_{y}:=\min \left\{d_{x}, d_{y}\right\}, \quad d_{x} \vee d_{y}:=\max \left\{d_{x}, d_{y}\right\}$. We cannot do this, because curvature $\mathcal{K}(x,y)$ is \textbf{not symmetric} with respect to $x$ and $y$, which is an assumption they make for this. However, we will make an assumption for now that their general proof structure works. They show for undirected graphs with the undirected costs that: 
\begin{figure}
    \centering
    \includegraphics[width=0.75\linewidth]{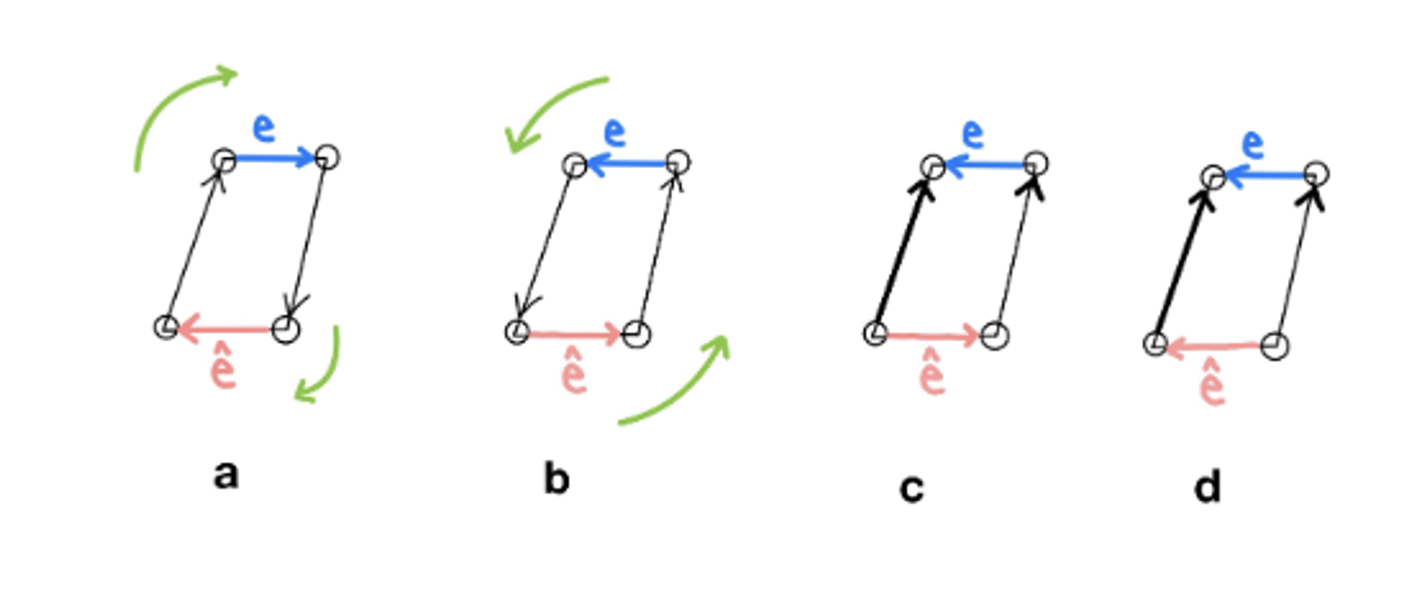}
    \caption{Four possible variations of a directed and oriented 4-cycle, adding more possible cyclic variations compared with the undirected case.}
    \label{fig:directed-4cycle-variants}
\end{figure}
$$
\begin{aligned}
W_{1}\left(m_{x}, m_{y}\right) \leq & \frac{1}{d_{x}} \times 1+\frac{1}{d_{y}} \times 1+\left(\frac{1}{d_{y}}-\frac{1}{d_{x}}\right) \times \sharp(x, y) \times 2 \\
& +\left[1-\frac{1}{d_{x}}-\frac{1}{d_{y}}-\left(\frac{1}{d_{y}}-\frac{1}{d_{x}}\right) \times \sharp(x, y)-\frac{1}{d_{x}} \sharp(x, y)\right] \times 3 \\
= & 3-\frac{2}{d_{x}}-\frac{2}{d_{y}}-\frac{\sharp(x, y)}{d_{y}}-\frac{2 \sharp(x, y)}{d_{x}} .
\end{aligned}
$$
In the directed graph setting where $|Ce| = 3$ for any given directed cycle, I will divide this into the two cases above: 
\begin{enumerate}
    \item $\overrightarrow{uv} \in (uvz)$ directed 3-cycle.
Then,
$$
\begin{aligned}
W_{1}\left(m_{x}, m_{y}\right) \leq & \frac{1}{d_{x(in)}} \times 1+\frac{1}{d_{y(out)}} \times 1+\left(\frac{1}{d_{y(out)}}-\frac{1}{d_{x(in)}}\right) \times \sharp(C_e = 3) \times 2 \\
& +\left[1-\frac{1}{d_{x(in)}}-\frac{1}{d_{y(out)}}-\left(\frac{1}{d_{y(out)}}-\frac{1}{d_{x(in)}}\right) \times \sharp(C_e = 3)-\frac{1}{d_{x(in)}} \sharp(C_e = 3)\right] \times 3 \\
= & 3-\frac{2}{d_{x(in)}}-\frac{2}{d_{y(out)}}-\frac{\sharp(C_e = 3)}{d_{y(out)}}-\frac{2 \sharp(C_e = 3)}{d_{x(in)}} 
\end{aligned}
$$
    \item $\overrightarrow{vu} \in (uvz)$ directed 3-cycle.
Then, 
$$
\begin{aligned}
W_{1}\left(m_{y}, m_{x}\right) \leq & \frac{1}{d_{x(in)}} \times 1+\frac{1}{d_{y(out)}} \times 1+\left(\frac{1}{d_{y(out)}}-\frac{1}{d_{x(in)}}\right) \times \sharp(C_e = 3) \times 3 \\
& +\left[1-\frac{1}{d_{x(in)}}-\frac{1}{d_{y(out)}}-\left(\frac{1}{d_{y(out)}}-\frac{1}{d_{x(in)}}\right) \times \sharp(C_e = 3)-\frac{1}{d_{x(in)}} \sharp(C_e = 3)\right] \times 4 \\
= & 4-\frac{3}{d_{x(in)}}-\frac{3}{d_{y(out)}}-\frac{\sharp(C_e = 3)}{d_{y(out)}}-\frac{3 \sharp(C_e = 3)}{d_{x(in)}}
\end{aligned}
$$
\end{enumerate}
We are of course interested in generalizing this to different types of cycles. Furthermore, we must ask ourselves if we have accounted for and removed all assumptions of symmetry from Jost, Liu. Below again are the assumptions we make:
\begin{itemize}
    \item $d_x(in), d_y(out) > 0$
    \item We are exclusively looking at 3-cycles of either possible directions
\end{itemize}

\begin{figure}
    \centering
    \includegraphics[width=1.0\linewidth]{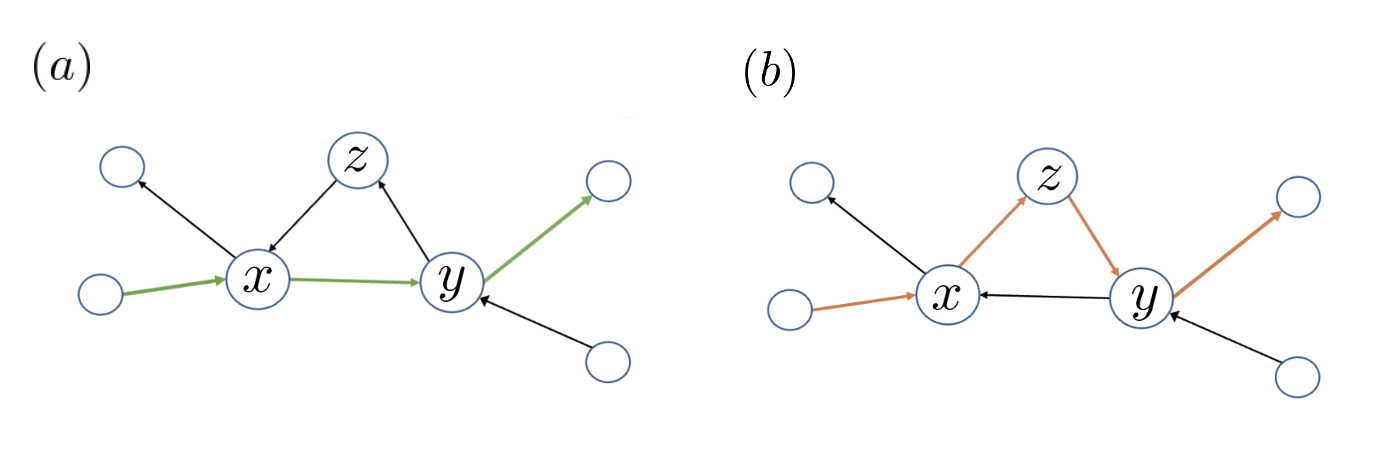}
    \caption{Mass distribution rules must change from Jost and Liu.}
    \label{fig:mass-distribution-rules}
\end{figure}
Interestingly, observe now that cases (a) and (b) in Figure 4.8 are really the same case but we are flipping $x$ and $y$. This brings the question: should we consider these two different directions, and how to we consider the curvature contribution of the non-shortest length path if it flows in our desired direction from $x$ neighbors to $y$ neighbors? In the case of Jost and Liu's work, they do not because they treat $W(x,y)$ as symmetric, in other words $W(x,y) = W(y,x)$. But in the directed case, of course we have $W(x,y) \neq W(y,x)$. So we should not be general and say for every pair $x,y$ but instead, for every edge $x \rightarrow y$, if this directional edge exists then we will only compute the cost of moving mass from $m_x$ to $m_y$. For this reason, breaking down the triangle counts by whether $|C|_e = 1$ or $|C|_e = 1$ is not sufficient. For the $|C|_e = 1$ case, there must be a directed path for us in the correct direction for us to consider the curvature contributions of the non-primary edges (non-$xy$ edges).

We will move on for now from cycles to discuss considerations for trees. Recall now the work of Jost and Liu where trees attain this lower bound on directed graphs:
$$
\kappa(x,y) \geq -2 (1 - \frac{1}{d_x} - \frac{1}{d^y})_{+}
$$
A consideration to make when working with directed trees is that we there exist in-branching, our-branching, and mixed-branching trees in which there is flow from the root word inwards, root node outwards, or a mix of flow. For research on Ollivier-Ricci curvature for directed trees we suggest first formalizing a notion and theoretical bounds pertaining separately to  true in-branching and out-branching trees respectively. Returning to the bound above, observe that this curvature value is equal to zero unless $d_x \geq 1$ and $d_y \geq 1$. 
\begin{figure}
    \centering
    \includegraphics[width=0.5\linewidth]{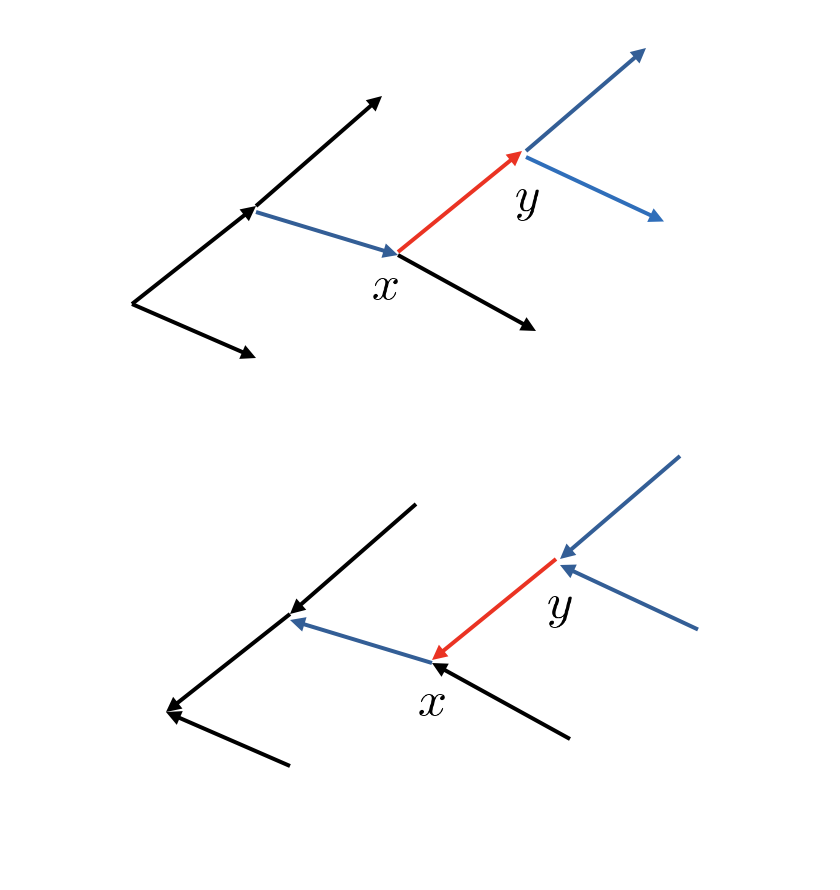}
    \caption{Directed out-branching tree (above) and in-branching tree (below).}
    \label{fig:directed-trees-branching}
\end{figure}

To begin with an out-branching tree, recall that all nodes aside from the root have exactly one incoming edge, one or more outgoing edge if not a leaf node, and finally no outgoing edges for leaf nodes. Consider the typical transport costs involved in the transport plan of (1) moving mass of $\frac{1}{d_x(in)}$ from y to y neighbors and $d_y(out)$ from x neighbors to x. Can we for example consider the singular out-neighbor of $x$ as a gap to fill? If so, the cost is two. 
Based on on Jost and Liu's argument for undirected graphs we could have have something of the form below if $d_x(out) = 2$ for example $1 - \frac{1}{d_y (out)} \geq \frac{1}{2d_x(in)}$
or an analogous extension. Though this may not be robust and model trees in other sample cases. For further investigation, we could follow the method of Jost-Liu of assigning Lipschitz functions to each vertex depending on their relationship within the neighborhoods of a transport plan. An open research question is how to reliably extend the combinatorial bounds of Jost-Liu to directed trees and cycles, and therefore there are countless other possible approaches. For example, for any number of proofs on theoretical bounds we could make the following assigments. 
\[
f(z) =
\begin{cases}
    0, & \text{if } z \in \mathcal{N}_{out}(y) z \neq x; \\
    1, & \text{if } z = y; \\
    2, & \text{if } z = x; \\
    3, & \text{if } z \in \mathcal{N}_{out}(x), \, z \neq x.
\end{cases}
\]

For the tree-based approach, observe that the key-difference between in-branching trees and out-branching trees is this: for all vertices of an out-branching tree the in-degree of a vertex is always exactly one, but the out-degree can be any number of vertices greater than or equal to one. For all vertices of an in-branching tree the in-degree of a vertex is greater or equal to one, but the out-degree can only be exactly one. Such limitations can actually be leveraged to make precise theoretical bounds for trees of these forms in the style of Jost-Liu through restrictions of the transport plans based on set in-degrees and out-degrees of a given vertex. 
Another possible approach, which we will not discuss in detail as it is largely beyond the scope of this thesis, is to exploit the work that has currently been done in the field of signed graphs and signed networks, as well as utilizing specific graph Laplacians used in these fields and other key results on graph structure and balance of structure that could be useful in simplifying some of the primary concerns of this chapter. For example, a research area to explore is to use ideas of structural balancedness and magnetic Laplacians as presented by Tian, Lambiotte in \cite{tian2024structural}.

In summary, in this section we explore considerations for a simplified directed Ollivier-Ricci curvature concept that does not involve a lazy-random walk and which is seeks to stay true to the original Ollivier-Ricci cruvature notion on graphs in a singular direction, with the understanding that asymmetry and directionality make the 1-Wasserstein metric on these newly defined measures such that it is not a true distance metric but can still capture essential properties and result ultimately in a directed Ollivier-Ricci curvature calculation at the edge-level that accurately captures flow. We explore two primary cases in which directionality must be handled with great care: cycles and trees. This is especially due to the fact of high positive curvature contributions of cycles and high negative contributions of trees, it is important to consider how mass transport is altered in the directed case. We discuss \textit{effective length} as a method of simply calculating and parsing different cycle-types, and give some ideas for how to then construct analogous theoretical bounds based on Jost-Liu with such contraints in mind. Then we moved the discussion to the case of trees and discussed out in-branching and out-branching adds additionally layers of complexity. Addressing these considerations and ideas, through experimental validation of directed curvature behavior, and theoretical work on bounds, for example, provides a set of challenging open research questions, which we will for now label as a fruitful area for future work. 
\chapter{Curvature-informed Algorithms on Graphs and Networks}

\begin{figure}
    \centering
    \includegraphics[width=0.75\linewidth]{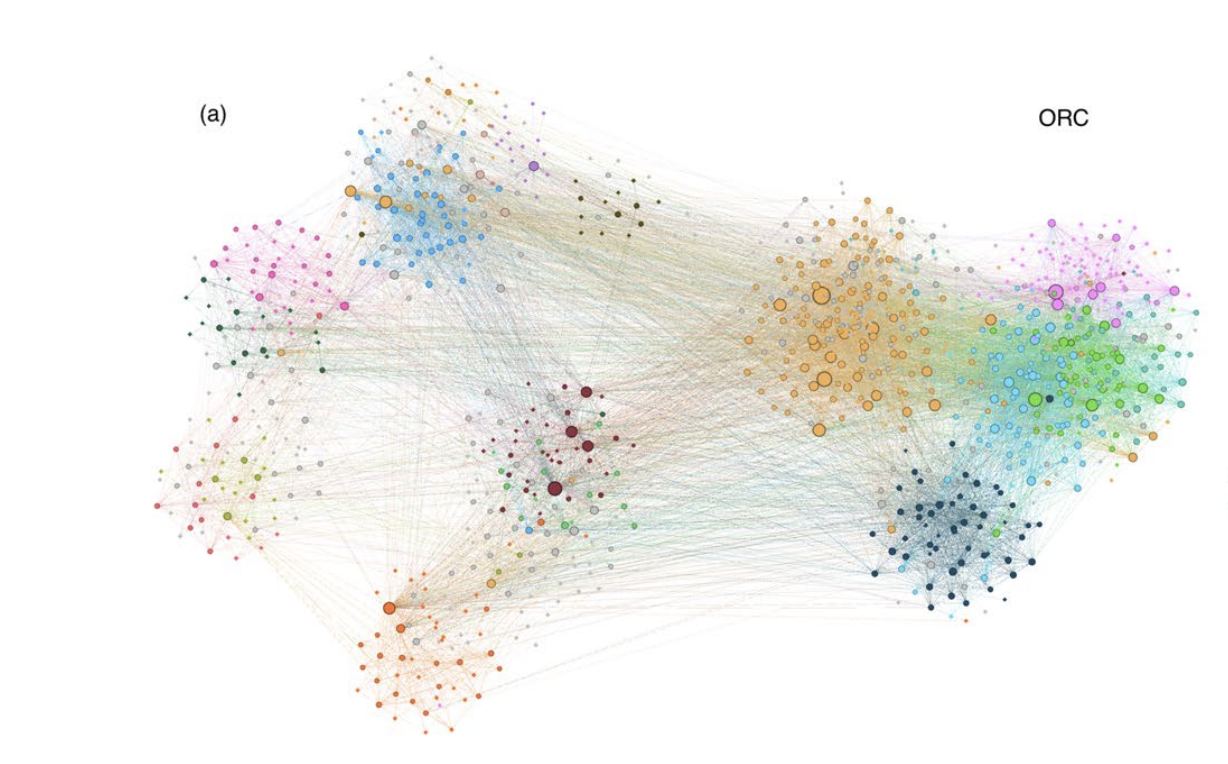}
    \caption{From Sia et al.~\cite{SiaJayson2019OCMt}, who use an Ollivier--Ricci curvature-based community detection algorithm to uncover community structures in an undirected drug--drug interaction dataset obtained from the DrugBank v4.1 database.}
    \label{fig:sia-network}
\end{figure}

The motivation for the extensive discussion of graphs and localized Ricci curvature on graphs in previous chapters is that we can extend these results to network for achieving various computational tasks. We will discuss in this chapter how Ricci curvature can be used to detect communities within a network. We will also introduce how Ricci curvature can assist with the task of graph or network \textit{rewiring} in which edges are added or removed to improve message-passing with graph neural networks. These areas have been highly studied for undirected graphs and networks. This thesis seeks to extend some of these classical results and algorithm for the purpose of using Ollivier Ricci curvature of Directed graphs to achieve improved community detection or graph rewiring, for example. These results will be introduced in the following chapter. This chapter will focus on introducing the Ricci curvature-informed algorithms for completing these computational tasks, and furthermore, why these algorithms matter in network science and machine learning applications.

An extensive amount of real-world network data is \textit{directed} in nature. For example, imagine a network of the flights taken between all global airports where each node of the network is an airport and each arc represents a plane's directed path from one starting destination to a final destination. This is just like a massive directed graph in which the vertices possess information, which is what makes nodes of a network unique from vertices of a graph. Consider another example. The internet is actually an immensely large directed network, where each page on the world wide web is a node, and each link is considered an arc in with points from one web page to another. The directed information is essential because a link from the Harvard Math department webpage to a professor's website is distinct from the link from a professor's website back to the Harvard Math homepage - one is called a backlink and one is a forward link. The internet as we know it, especially search engines like Google's PageRank, are based on the directionality of the network arcs. A social media network is another example of a directed network. Individual users each form a node and their "following" relationship to each other such as can be modeled by arcs pointing from one user to another, and there we will be two opposite arcs if the two users both follow one another. Of course, such networks are not limited to social or economic behaviors as we have mentioned so far. Biological networks such as gene regulatory networks (GRNs) can also be directed. In the case of GRNs, each node of the network is a specific gene, and one gene is pointing to another in the network if it is responsible in that gene's regulation. That said, it is my hope than any reader is now convinced of the fascinatingly ubiquitous nature of both directed and undirected networks in our world today.

One problem in network science, and especially in the subfields of curvature-based community detection or graph rewiring, it is often common for researchers to convert networks from directed to undirected for ease of computation and for the purpose of using established undirected network algorithms. For some use-cases this is a reasonable and computationally efficient decision. However, in other cases, there is potentially a loss of rich information regarding the directed connections between nodes. A future direction to address this issue is to continue research on curvature-based algorithms which explicitly use novel directed graphs notions of curvature.

\subsection*{Curvatrue-based Community Detection}
We will now discuss the first of two classic problems in network science research. This is the act of being able to detect community groups or clusters within a network. For example, in a social network that represents all college students in Boston, it is expected that communities would form at each university such that the users at Harvard, for example, form a highly connected group or community. Uncovering these communities is a task  called \textit{community detection}, and when it is performed on simple graphs, it is also referred to as \textit{graph clustering}. The overarching goal of community detection is to infer the structure underlying a network through the clustering of vertices that have similar connectivity patterns. This task can also be used to uncover communities within a labeled network in which nodes belong to specific groups or label classes, and community detection allows the uncovering of such groups. There are different types of connectivity patterns that a network may have that can be uncovered through community detection tasks. For example, some networks may be characterized by nodes being more highly connected to nodes of the same community than to nodes outside of the community, and there may also be the case where the opposite is true. Curvature-based community detection tends to answer the question of identifying or detecting communities that display highly-connected structure, but this is not by definition necessary. 

\begin{figure}
    \centering
    \includegraphics[width=0.7\linewidth]{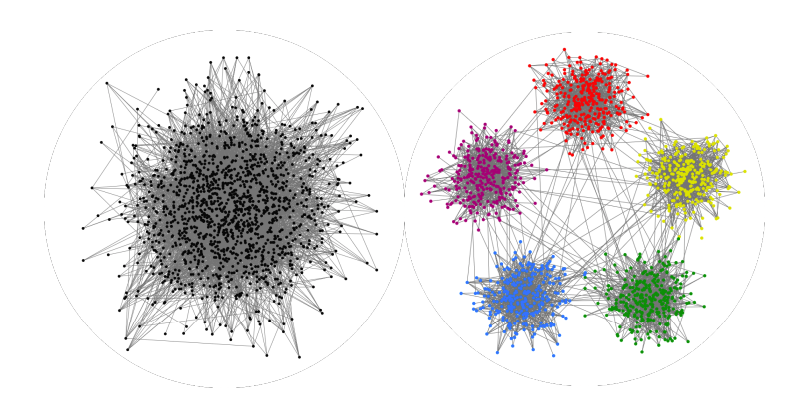}
    \caption{Figure by Emmanuel Abbe \cite{abbe2018community}. Left and right graphs are identical. These two graphs have been created with the
SBM model with 1000 vertices and  5 specified communities (balanced),  intra-community probability of 1/50 and inter-community probability of 1/1000 for edge formation. }
    \label{fig:sbm-intro}
\end{figure}

\begin{figure}
    \centering
    \includegraphics[width=1.0\linewidth]{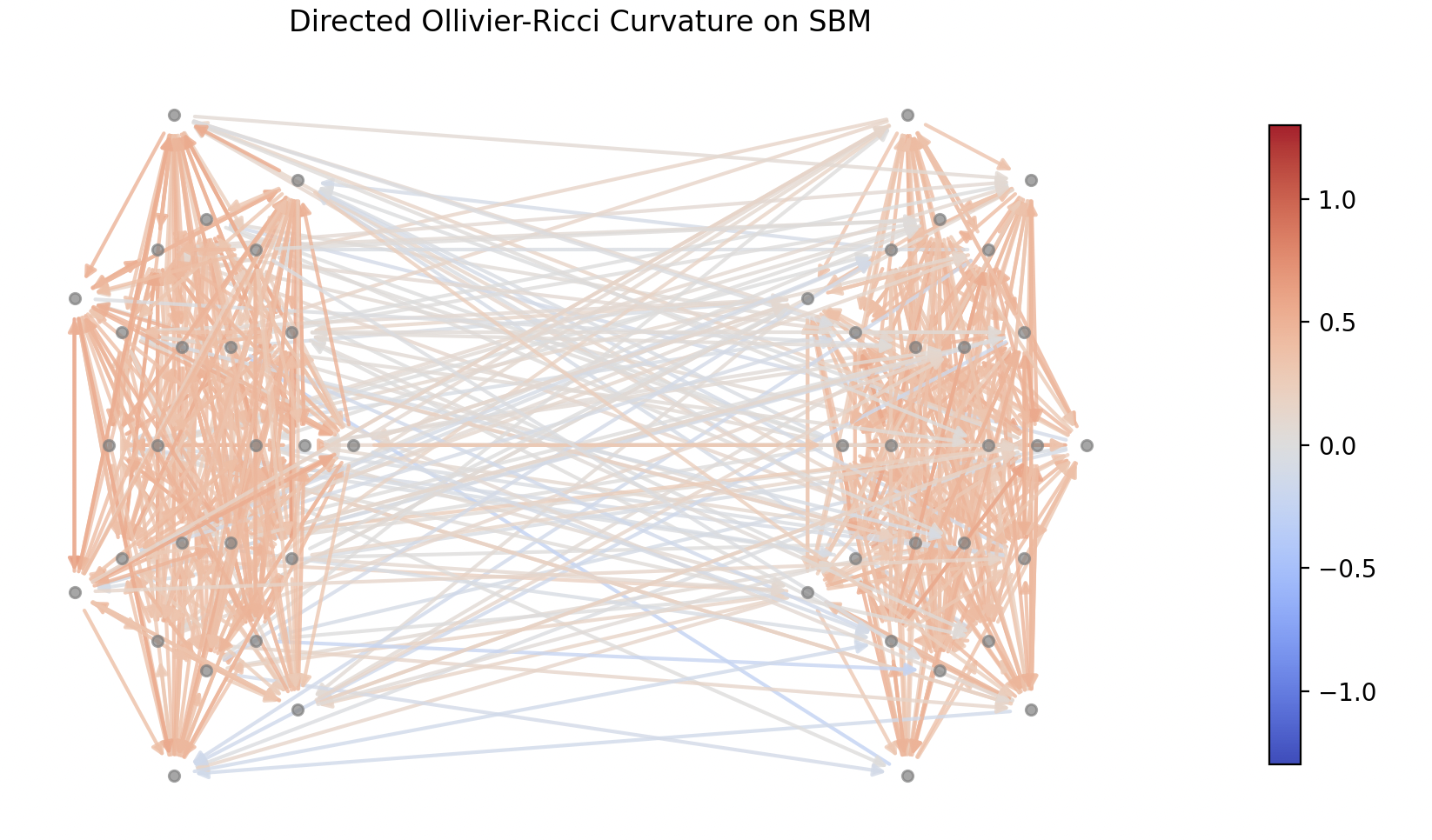}
    \caption{Example of a directed stochastic block model with 50 nodes, specified inter- and intra-community edge probabilities, and Ollivier--Ricci curvature plotted on edges.}
    \label{fig:diSBM}
\end{figure}

There are numerous algorithm for community detection, with the discovery of novel algorithms being an active research area. Such algorithms are often tested by creating a randomly generated graph with a set number of communities, called a \textit{stochastic block model} generated as follows:  all $n$ vertices are independently assigned to one of $k$ communities according to a probability $p$, and the community assignments for all $n$ vertices form the ground truth community values for the graph. Then edges between any pair of vertices are created independently according to $k \times k$ symmetric connectivity matrix $W$ in which the entries correspond to probability of an edge between each pair of vertices. These probabilities can be adjusted. As an example, in the creation of the graph one might set the edge formation probability between a pair of vertices in the same community as higher than the probability of edges outside of the community. For example, Figure~\ref{fig:sbm-intro} demonstrates an example of an SBM-generated graph with five communities in which intra-community and inter-community edge probabilities differ, allowing for community structures to form based on connectivity, and \ref{fig:diSBM} depicts an two-block, or two-community SBM model generated using directed edges only, with plotted curvature on each edge. Observe that there is a beneficial property as discussed throughout this thesis that an increase in triangles or connectivity results in more positive Ollivier-Ricci edge values in clusters and more negative edge curvature values on edges that connect communities. This discrepancy can be used to identify communities using algorithms that iteratively splice bridges for example, or others which us alternative algorithms to identify when a community or cluster separates from neighbors connected by only a bridge edge. For example, in Algorithm 2, we cite the clustering algorithm of  Tian, Lubberts, and Weber, \cite{TianYu2023CCoG}, the authors iterate through the edges of a graph and update an assigned weight to each edge based on Ricci which captures valuable curvature information, and then uses the results of this iteration to isolate or detect highly connected communities, and storing these community labels. This is a simplified version of their method, but captures how edge iteration and updating of some sort of curvature characteristic can allow for one to determine edge-level curvature and connectivity behavior of a large network using a graph curvature notion such as Ollivier-Ricci curvature. Other such algorithms include that by \cite{SiaJayson2019OCMt} used for the community detection in Figure \ref{sia-network}, which is presented in Algorithm 3. Here, the authors calculate the Ollivier-Ricci curvature for each edge in the graph or network and simply remove the edge is this edge-list with the most negative curvature, and then continue the process of re-calculating ORC and removing most negative edges, as curvature of surrounding edges changes when there are perturbation to graph structure or connecting neighbors. This method is simple and relies on the notion that bridges between communities will be negative, and communities themselves will be positive in curvature overall.

\singlespacing
\begin{algorithm}[H]
\caption{Curvature-based Clustering via Ricci Flow from \cite{TianYu2023CCoG}}
\begin{algorithmic}[1]
\Require Graph $G = (V, E, W_0)$, hyperparameters $\nu$ (step size, flow), $\epsilon$ (tolerance), $\epsilon_d$ (drop threshold), $T$ (number of iterations)
\For{$t = 1, \ldots, T$}
    \For{ $\{u,v\} \in E$}
        \State Compute curvature $\kappa(\{u,v\})$.
        \State Evolve weight under Ricci flow: 
        \[
        w^t(\{u,v\}) \leftarrow (1 - \nu \cdot \kappa(\{u,v\})) d_G(u,v)
        \]
    \EndFor
    \State Renormalize edge weights $W^t$: 
    \[
    w^t(\{u,v\}) \leftarrow \frac{|E| d_G(u,v)}{\sum_{\{u',v'\} \in E} d_G(u',v')}
    \quad \text{for all } \{u,v\} \in E
    \]
\EndFor
\State Construct cut-off points $\{x_0, x_1, \ldots, x_{n_f}\}$
\State Initialize $\mathbf{l} \in \mathbb{R}^{|V|}$ (list of node labels), $Q_{-1}, Q_{\text{best}} := \epsilon$
\For{$i = 0, \ldots, n_f$}
    \State Construct $\tilde{G}_i = (V, E_i, W^T_i)$ with $E_i = \{\{u,v\} : w^T(\{u,v\}) > x_i\}$ and $W^T_i = W^T|_{E_i}$
    \State Determine connected components of $\tilde{G}_i$. Assign node labels by components.
    \State Compute the modularity $Q_i$ of the resulting community assignment.
    \If{$\frac{Q_i - Q_{i-1}}{Q_i} > \epsilon_d$}
        \State Store label assignments in $\mathbf{l}$. $Q_{\text{best}} = Q_i$
    \EndIf
\EndFor
\If{$Q_{\text{best}} > \epsilon$}
    \Return $\mathbf{l}$
\EndIf
\end{algorithmic}
\end{algorithm}

\begin{algorithm}
\caption{Ollivier-Ricci Curvature based method for Community Detection from \cite{SiaJayson2019OCMt}}
\begin{algorithmic}[1]
\State \textbf{Input:} A graph object $G(V, E, \rho)$ with a list of nodes, $V$, and a list of edges, $E$.
\State \textbf{Output:} A graph object $G'(V, E, \rho)$ similar to $G$ but with additional node properties indicating community label.
\State $G' \gets G$ with the Ollivier-Ricci curvature calculated for all edges
\While{there exists a negative edge curvature in $G'$}
    \State Remove the most negatively curved edge
    \State Re-calculate the Ollivier-Ricci curvature for the affected existing edges in $G'$
\EndWhile
\State \textsc{PreferentialAttachment}$(G', \textit{number\_of\_communities}, \textit{minimum\_community\_size})$
\State Label each node with a unique community label according to its membership to a particular graph component
\State \textbf{return} $G'$
\end{algorithmic}
\end{algorithm}
\section*{Curvature-based Graph Rewiring}
\begin{figure}
    \centering
    \includegraphics[width=0.65\linewidth]{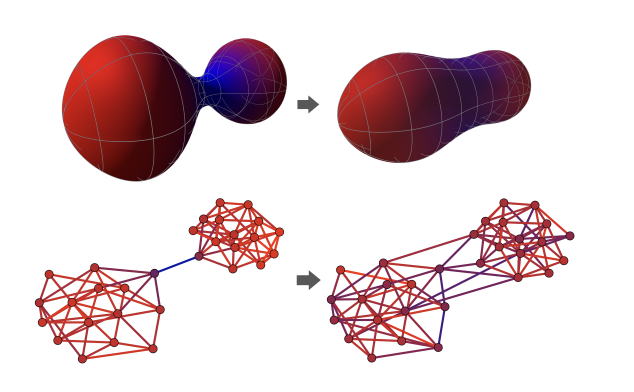}
    \caption{From \cite{topping2021understanding}: the top of this figure demonstrates how reduction of a bottleneck may be prevented through increased volume. The paper shows that bottlenecks can be reduced on graphs by adding edges between communities of positive curvature for improved message passing in graph neural networks. Negative Ricci curvature is shown in blue and positive Ricci curvature in red.}
    \label{fig:bottleneck-curvature}
\end{figure}

\doublespacing
This second type of curvature-based algorithm on graphs and networks depends on the understanding of the basics of graph neural networks (GNNs), which we refer the reader to learn the basics of in the textbook \textit{Graph Representation Learning} by William L. Hamilton (2020). The problem which curvature-basd rewiring seeks to address is the  problem of \textit{over-smoothing} and \textit{over-squashing} in GNNs. \textit{Over-smoothing} is 

\doublespacing
This second type of curvature-based algorithm on graphs and networks addresses foundational issues encountered in graph neural networks (GNNs), namely \textit{over-smoothing} and \textit{over-squashing}. For background on GNNs, we refer the reader to \textit{Graph Representation Learning} by William L. Hamilton (2020).

\textit{Over-smoothing} occurs when node representations become indistinguishable from one another as layers of the GNN increase. This leads to a loss of discriminative power, as the features of distant nodes converge to similar values due to excessive message passing. In effect, the network learns representations that are too smooth across the graph, washing out local information. Two nodes which share membership of a highly connected cluster for example, will experience over-smoothing.  \textit{Over-squashing}, on the other hand, arises when information from an exponentially growing neighborhood must be compressed into fixed-size node embeddings. This bottleneck can prevent distant, but relevant, information from properly influencing a node's representation. It often occurs between nodes in graphs with long-range dependencies or bottlenecks in the graph topology. Curvature-based rewiring techniques aim to mitigate these issues by modifying the graph structure in a way that improves the flow of information, reducing bottlenecks but also reducing an excess of flow or mass in clustered regions when it comes to GNN message-passing. In particular, these methods exploit notions of curvature like Ollivier-Ricci to identify and adjust connections in the graph, promoting more efficient aggregation paths and improving model expressivity. For more infromation on these concepts see \cite{topping2021understanding} and \cite{giraldo2023trade}, for example. We will discuss specific references in the remainder of this subsection as well. 

Various algorithms are established in the literature for rewiring graphs to systematically add additional edges between clusters so nodes across these clusters experience less over-squashing. Analogously edges can be systematically removed to address over-smoothing. We include an example algorithm used for augmented Forman Ricci curvature graph rewiring by Fesser, Weber \cite{FesserLukas2023MOaO} that uses edge curvature from all edges in a graph to systematically remove added if that fall within a certain curvature range, an add edges to the edge set by selecting appropriate neighborhoods for their connecting vertices. Note that augmented Forman Ricci curvature displays different behavior than Ollivier-Ricci curvature, but  analogous rewiring algorithms for Ollivier-Ricci are an active area of research. This has so far been done by \cite{topping2021understanding}, \cite{attali2024rewiring}, \cite{nguyen2023revisiting-rewiring} as examples. Furthermore, we note that these methods to the directed case, as presented in a sample example in Figure 5.5., should take into account directionality of bridges added when performing a directed rewiring technique. To our knowledge, this has not been studied and is another open research question.

\begin{figure}
    \centering
    \includegraphics[width=1.0\linewidth]{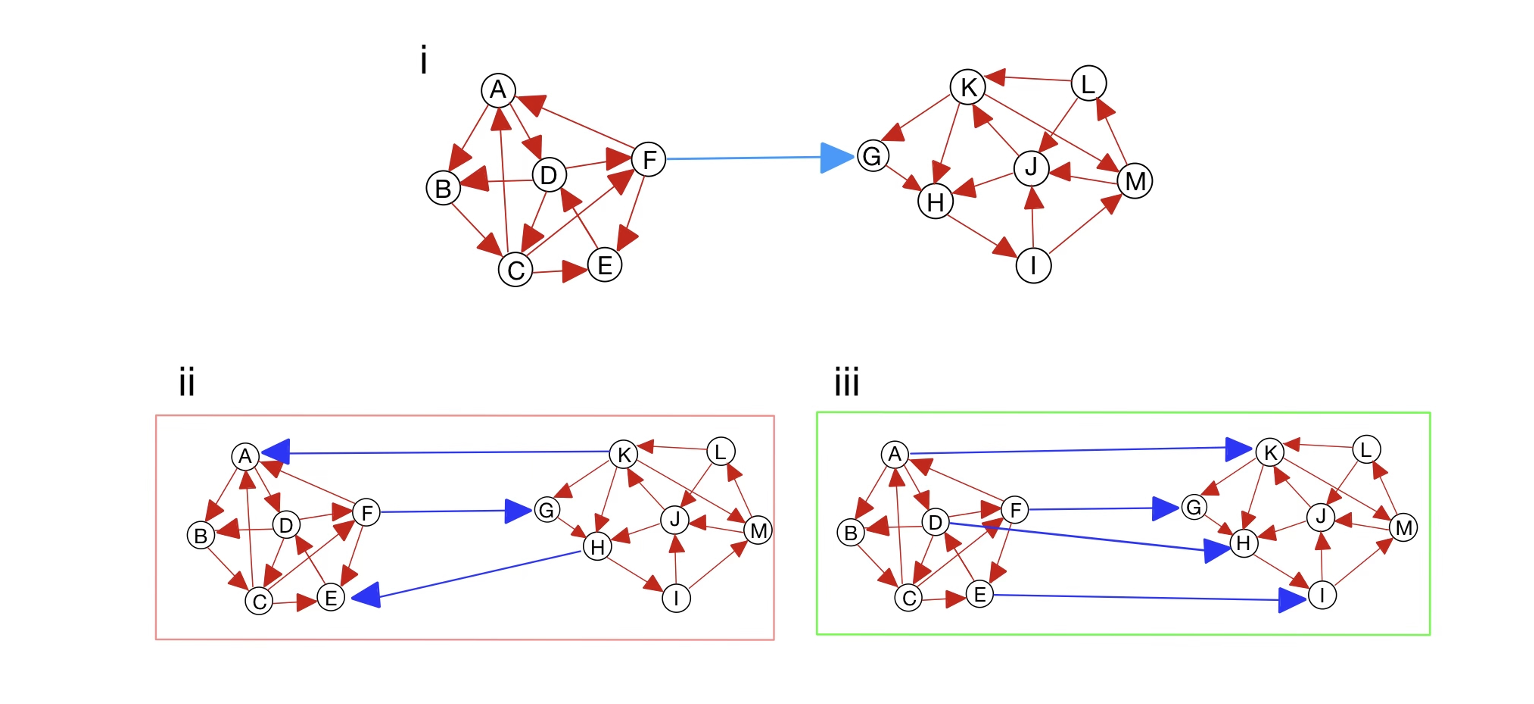}
    \caption{Graph rewiring in the directed case must consider directionality of added edges when reducing over-smoothing and over-squashing.}
    \label{fig:directed-rewiring}
\end{figure}

\begin{algorithm}
\caption{Augmented Forman Ricci Curvature Graph Rewiring from \cite{FesserLukas2023MOaO}}
\begin{algorithmic}[1]
\Require A graph $G(V, E)$, heuristic (Bool), \# edges to add $h$, \# edges to remove $l$
\For{each $e \in E$}
    \State compute $\mathcal{AF}_3(e)$
\EndFor
\State Sort $e \in E$ by $\mathcal{AF}_3(e)$
\If{heuristic == True}
    \State Compute lower threshold $\Delta_L$
    \For{each $(u, v) \in E$ with $\mathcal{AF}_3(u, v) < \Delta_L$}
        \State choose $w \in \mathcal{N}_u \setminus \mathcal{N}_v$ uniformly at random
        \State add $(w, v)$ to $E$
    \EndFor
    \State Compute upper threshold $\Delta_U$
    \State Remove all edges $(u, v)$ with $\mathcal{AF}_3(u, v) > \Delta_U$
\Else
    \State Find $h$ edges with lowest $\mathcal{AF}_3$ values
    \For{each of these edges}
        \State choose $w \in \mathcal{N}_u \setminus \mathcal{N}_v$ uniformly at random
        \State add $(w, v)$ to $E$
    \EndFor
    \State Find $k$ edges with highest $\mathcal{AF}_3$ values and remove them from $E$
\EndIf
\Return a rewired graph $G' = (V, E')$ with new edge set $E'$
\end{algorithmic}
\end{algorithm}

\setstretch{\dnormalspacing}

\listoffigures
\clearpage

\backmatter

\end{document}